\documentclass[12pt,a4paper]{article}
\usepackage{amsmath,amssymb,amsthm,graphicx,color}
\usepackage[left=1in,right=1in,top=1in,bottom=1in]{geometry}
\usepackage{cite}
\usepackage[british]{babel}
\usepackage{hyperref} 
\usepackage{enumitem}
\usepackage{tikz}
\usetikzlibrary{decorations.markings}

\hypersetup{
    colorlinks=false,
    pdfborder={0 0 0},
}

\newtheorem{theorem}{Theorem}[section]

\newtheorem{proposition}[theorem]{Proposition}
\newtheorem{lemma}[theorem]{Lemma}

\numberwithin{equation}{section}

\theoremstyle{definition}

\newtheorem{examplex}[theorem]{Example}

\newenvironment{example}
  {\pushQED{\qed}\examplex}
  {\popQED\endexamplex}

\theoremstyle{remark}
\newtheorem{remark}[theorem]{Remark}
\newtheorem{remarks}[theorem]{Remarks}
\newtheorem*{remark*}{Remark}

 \allowdisplaybreaks

\newcommand{\1}[1]{{\mathbf 1}{\{#1\}}}
\newcommand{\2}[1]{{\mathbf 1}_{#1}}

\newcommand{\R}{{\mathbb R}}
\newcommand{\Z}{{\mathbb Z}}

\newcommand{\ZP}{{\mathbb Z}_+}
\newcommand{\RP}{{\mathbb R}_+}

\DeclareMathOperator{\Exp}{\mathbb{E}}
\renewcommand{\Pr}{{\mathbb P}}

\newcommand{\tra}{{\scalebox{0.6}{$\top$}}}

\newcommand{\eps}{\varepsilon}

\newcommand{\re}{{\mathrm{e}}}
\newcommand{\rc}{{\mathrm{c}}}
\newcommand{\ud}{{\mathrm d}}

\newcommand{\cD}{{\mathcal D}}

\newcommand{\cF}{{\mathcal F}}

\newcommand{\cR}{{\mathcal R}}

\newcommand{\bc}{\beta_{\mathrm{c}}}

\makeatletter
\def\namedlabel#1#2{\begingroup  
    (#2)%
    \def\@currentlabel{#2}%
    \phantomsection\label{#1}\endgroup
}
\makeatother

\newlist{myenumi}{description}{10}
\setlist[myenumi]{labelindent=\parindent, leftmargin=*, align=left, itemsep=1pt, parsep=0pt}
\setlist[myenumi]{leftmargin=0pt}

\begin{document}

\title{Reflecting random walks in curvilinear wedges}
\author{Mikhail V.\ Menshikov\footnote{Durham University} \and Aleksandar Mijatovi\'c\footnote{University of Warwick and the Alan Turing Institute} \and Andrew R.\ Wade\footnotemark[1]}


\maketitle

\begin{center}
Dedicated to our colleague Vladas Sidoravicius (1963--2019)
\end{center}

\begin{abstract}
We study a random walk (Markov chain) in an unbounded planar domain whose boundary is described by two curves
of the form $x_2 = a^+ x_1^{\beta^+}$
and $x_2 = -a^- x_1^{\beta^-}$, with $x_1 \geq 0$.
In the interior of the domain, the random walk has zero drift and a given increment covariance matrix. From the vicinity of the upper and lower sections of the boundary, the walk drifts back into the interior at a given angle $\alpha^+$ or $\alpha^-$ to the relevant inwards-pointing normal vector.
Here we focus on the case where $\alpha^+$ and $\alpha^-$ are equal but opposite, which includes the case of normal reflection.
For $0 \leq \beta^+, \beta^- < 1$, we identify the phase transition between recurrence and transience, depending on the model parameters,
and quantify recurrence via moments of passage times. 
\end{abstract}

\medskip


\medskip


\section{Introduction and main results}
\label{sec:intro}

\subsection{Description of the model}
\label{sec:model}

We describe our model and then state our main results: see \S\ref{sec:literature}
for a discussion of related literature.
Write $x \in \R^2$ in Cartesian coordinates as $x = (x_1,x_2)$.
For parameters $a^+, a^- >0$ and $\beta^+,\beta^- \geq 0$,
 define, for $z \geq 0$, functions
 $d^+ ( z ) := a^+ z^{\beta^+}$ and  $d^- ( z ) := a^- z^{\beta^-}$.
Set
\[ \cD := \left\{ x \in \R^2 : x_1 \geq 0, \, - d^- (x_1 ) \leq x_2 \leq d^+ (x_1) \right\} .\]
Write $\| \, \cdot \, \|$ for the Euclidean norm on~$\R^2$.
For $x \in \R^2$ and $A \subseteq \R^2$, write $d(x,A) := \inf_{y \in A} \| x - y \|$ for the distance
from $x$ to $A$.
Suppose that there exist $B \in (0,\infty)$ and a subset $\cD_B$ of $\cD$ for which every $x \in \cD_B$ has $d(x, \R^2 \setminus \cD ) \leq B$.
Let $\cD_I := \cD \setminus \cD_B$; we call $\cD_B$ the \emph{boundary} and $\cD_I$ the \emph{interior}.
Set $\cD^\pm_B := \{ x \in   \cD_B : \pm x_2 > 0\}$
for the   parts of $\cD_B$ in the upper and lower half-plane, respectively.

Let $\xi := (\xi_0, \xi_1, \ldots)$ be a discrete-time, time-homogeneous Markov chain on state-space $S \subseteq \cD$.
Set $S_I := S \cap \cD_I$, $S_B := S \cap \cD_B$, and $S^\pm_B := S \cap \cD^\pm_B$.
Write $\Pr_{x}$ and $\Exp_{x}$ for conditional probabilities and expectations given $\xi_0 = x \in S$,
and suppose that $\Pr_x ( \xi_n \in S \text{ for all } n \geq 0 ) = 1$ for all $x \in S$.
Set $\Delta := \xi_1 - \xi_0$.
Then~$\Pr ( \xi_{n+1} \in A \mid \xi_n = x ) = \Pr_{x} ( x + \Delta \in A )$
for all $x \in S$, all measurable $A \subseteq \cD$, and all $n \in \ZP$.
In what follows, we will always treat vectors in $\R^2$ as column vectors.

We will assume that $\xi$ has uniformly bounded $p >2$ moments for its increments,
that in $S_I$ it 
has zero drift and a fixed increment covariance matrix, and that it 
\emph{reflects} in $S_B$, meaning it has drift away from $\partial \cD$ at a certain angle relative to the inwards-pointing normal vector.
In fact we permit perturbations of this situation that are appropriately small as the distance from the origin increases.
See Figure~\ref{fig:picture} for an illustration.


\begin{figure}
\begin{center}
\begin{tikzpicture}[domain=0:10, scale = 1.2]
\filldraw (0,0) circle (1.5pt);
\node at (-0.25,0) {$0$};
\draw[black, line width = 0.40mm]   plot[smooth,domain=0:10,samples=500] ({\x},  {(\x)^(1/2)});
\draw[black, line width = 0.40mm]   plot[smooth,domain=0:10,samples=500] ({\x},  {-(\x)^(1/2)});
\draw[black, line width = 0.20mm, dotted]   plot[smooth,domain=1:10,samples=50] ({((\x)+((1/2)*(1/(sqrt(\x)))))*(1/(sqrt((1)+(1/(4*(\x))))))},  {((\x)^(1/2)-1)*(1/(sqrt((1)+(1/(4*(\x))))))});
\node at (9,2.55) {$\cD_B^+$};
\node at (9,-2.55) {$\cD_B^-$};
\node at (9,0.5) {$\cD_I$};
\draw[black, line width = 0.20mm, dotted]   plot[smooth,domain=1:10,samples=50] ({((\x)+((1/2)*(1/(sqrt(\x)))))*(1/(sqrt((1)+(1/(4*(\x))))))},  {-((\x)^(1/2)-1)*(1/(sqrt((1)+(1/(4*(\x))))))});
\draw[black,->,>=stealth,dashed] (0,0) -- (10,0);
\node at (10.3, 0)       {$x_1$};
\draw[black,->,>=stealth,dashed] (0,-4) -- (0,4);
\node at (0,4.3)       {$x_2$};
\draw (4,2) -- (7,2.75);
\draw (1,1.25) -- (4,2);
\draw (4,2) -- (4.3,0.8);
\draw (3.8,1.95) -- (3.86,1.71);
\draw (4.06,1.76) -- (3.86,1.71);
\draw (4.5,1.5) arc (-45:-67.5:1);
\node at (8, 3.5)       {$x_2 = a^+ x_1^{\beta^+}$};
\node at (8, -3.5)      {$x_2 = -a^- x_1^{\beta^-}$};
\draw[black,->,>=stealth,dashed] (4,2) -- (5,1);
\node at (4.8, 1.6)       {$\alpha^+$};
\end{tikzpicture}
\end{center}
\caption{\label{fig:picture} An illustration of the model parameters, in the case where $\beta^+ = \beta^- \in (0,1)$.}
\end{figure}
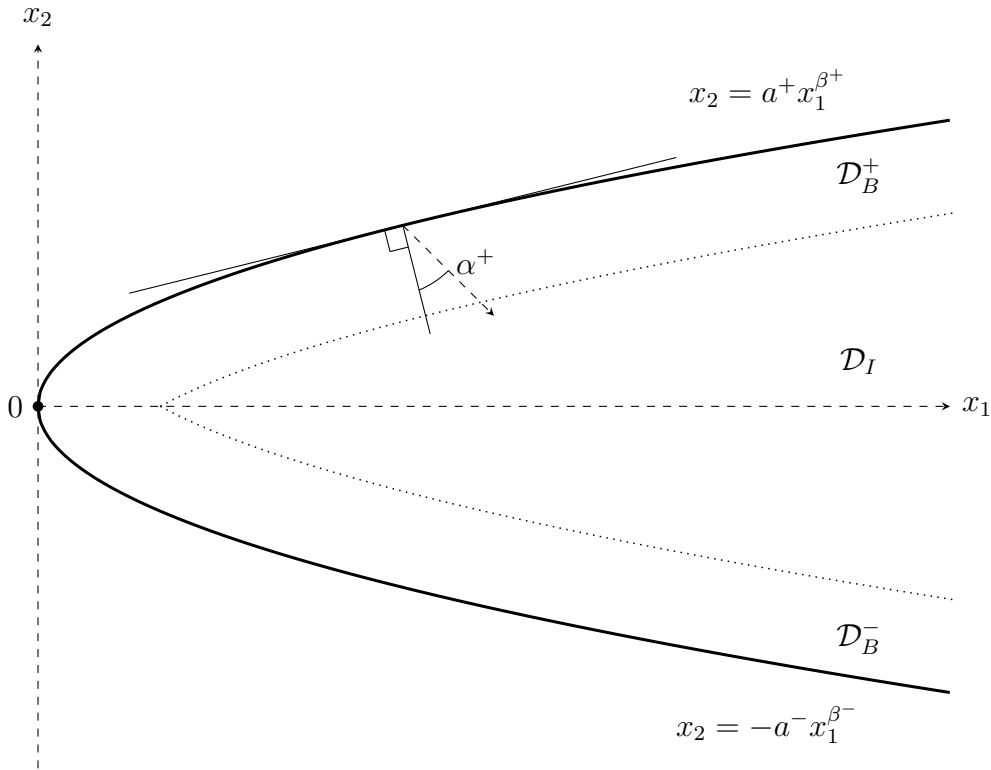
 
To describe the assumptions formally,
for $x_1 >0$ let $n^+ (x_1)$ denote the inwards-pointing unit normal vector to $\partial \cD$ at $(x_1, d^+ (x_1))$,
and let $n^- (x_1)$ be the corresponding normal at $(x_1, - d^- (x_1))$; then $n^+(x_1)$ is a scalar multiple of $(a^+ \beta^+ x_1^{\beta^+-1}, -1)$,
and $n^-(x_1)$ is a scalar multiple of $(a^- \beta^- x_1^{\beta^--1}, 1)$.
Let $n^+ (x_1, \alpha)$ denote the unit vector obtained by rotating $n^+ (x_1)$ by angle~$\alpha$ anticlockwise.
Similarly, let $n^- (x_1, \alpha)$ denote the unit vector obtained by rotating $n^- (x_1)$ by angle~$\alpha$ clockwise.
(The orientation is such that, in each case, reflection at angle $\alpha <0$ is pointing on the side of the normal towards $0$.)

 We write $\| \, \cdot \, \|_{\mathrm{op}}$ for the matrix (operator) norm defined by
$\| M \|_{\mathrm{op}} := \sup_{u} \| M u \|$, where the supremum is over all unit vectors $u \in \R^2$.
We take $\xi_0 = x_0 \in S$ fixed, and impose the following assumptions for our main results.

\begin{description}
\item\namedlabel{ass:non-confinement}{N} 
Suppose that $\Pr_x ( \limsup_{n \to \infty} \| \xi_n \| = \infty ) =1$ for all $x \in S$.
\item\namedlabel{ass:bounded-moments}{M$_p$} 
There exists $p >2$ such that 
\begin{equation}
\label{eq:p-moments} 
\sup_{x \in S} \Exp_x ( \| \Delta \|^p ) < \infty .\end{equation}
\item\namedlabel{ass:zero-drift}{D}
We have that
 $\sup_{x \in S_I : \| x \| \geq r } \| \Exp_x \Delta \| = o (r^{-1})$ as $r \to \infty$.
\item\namedlabel{ass:reflection}{R}
There exist angles $\alpha^\pm \in (-\pi/2,\pi/2)$ and  functions $\mu^\pm : S^\pm_B \to \R$ 
with $\liminf_{\| x \| \to \infty} \mu^\pm (x) >0$, such that, as $r \to \infty$,
\begin{align}
\label{eq:reflection+}  \sup_{x \in S_B^+ : \| x \| \geq r } \| \Exp_x \Delta - \mu^+ (x) n^+(x_1,\alpha^+) \| & = O ( r^{-1}) ; \\
\label{eq:reflection-}  \sup_{x \in S_B^- : \| x \| \geq r } \| \Exp_x \Delta - \mu^- (x) n^-(x_1,\alpha^-) \| & = O( r^{-1} ) 
.\end{align} 
\item\namedlabel{ass:covariance}{C}
There exists a positive-definite, symmetric $2 \times 2$ matrix $\Sigma$ for which 
\[ \lim_{r \to \infty} \sup_{x \in S_I : \| x \| \geq r } \bigl\| \Exp_x ( \Delta \Delta^\tra ) - \Sigma \bigr\|_{\mathrm{op}} = 0 . \]
\end{description}
We write the entries of $\Sigma$ in~\eqref{ass:covariance} as
\[ \Sigma = \begin{pmatrix} \sigma_1^2 & \rho \\
\rho & \sigma_2^2 \end{pmatrix} .\]
Here $\rho$ is the asymptotic increment covariance, and, 
since $\Sigma$ is positive definite, $\sigma_1 >0$, $\sigma_2 > 0$, and $\rho^2 < \sigma_1^2 \sigma_2^2$.

To identify the critically recurrent cases, we need slightly sharper control of the error terms in the drift assumption~\eqref{ass:zero-drift} and covariance assumption~\eqref{ass:covariance}.
In particular, we will in some cases impose the following stronger versions of these assumptions:
\begin{description}
\item\namedlabel{ass:zero-drift-plus}{D$_+$}
There exists $\eps >0$ such that
 $\sup_{x \in S_I : \| x \| \geq r } \| \Exp_x \Delta \| = O ( r^{-1-\eps} )$ as $r \to \infty$.
\item\namedlabel{ass:covariance-plus}{C$_+$}
There exists $\eps >0$ and a positive definite symmetric $2 \times 2$ matrix $\Sigma$ for which 
\[ \sup_{x \in S_I : \| x \| \geq r } \bigl\| \Exp_x ( \Delta \Delta^\tra ) - \Sigma \bigr\|_{\mathrm{op}} = O ( r^{-\eps} ) ,
\text{ as } r \to \infty. \]
\end{description}
Without loss of generality, we may use the same constant $\eps>0$ for both~\eqref{ass:zero-drift-plus} and~\eqref{ass:covariance-plus}.

The non-confinement condition~\eqref{ass:non-confinement}
ensures our questions of recurrence and transience (see below) are non-trivial, and 
is implied by standard irreducibility or ellipticity conditions: see~\cite{mpw} and the following example.
  
\begin{example}
Let $S = \Z^2 \cap \cD$, and take $\cD_B$ to be the set of $x \in \cD$
for which $x$ is within unit $\ell_\infty$-distance of some $y \in \Z^2 \setminus \cD$. Then~$S_B$
contains those points of $S$ that have a neighbour outside of $\cD$, and $S_I$ consists of those
points of $S$ whose neighbours are all in $\cD$. If $\xi$ is irreducible on $S$, then~\eqref{ass:non-confinement} holds
(see e.g.~Corollary 2.1.10 of~\cite{mpw}). If $\beta^+  >0$, then, for all $\| x\|$ sufficiently large,
every point of $x \in S_B^+$ has its neighbours to the right and below in $S$, so if $\alpha^+  =0$, for instance,
we can achieve
the asymptotic drift  required  by~\eqref{eq:reflection+} using only nearest-neighbour jumps if we wish; similarly in $S_B^-$.
\end{example}

Under the non-confinement condition~\eqref{ass:non-confinement}, the first 
question of interest is whether $\liminf_{n \to \infty} \| \xi_n \|$ is finite or infinite.
We say that $\xi$ is \emph{recurrent} if there exists $r_0 \in \RP$ for which $\liminf_{n \to \infty} \| \xi_n \| \leq r_0$, a.s.,
and that $\xi$ is \emph{transient} if  $\lim_{n \to \infty} \| \xi_n \| = \infty$, a.s. The first main aim of this
paper is to classify the process into one or other of these cases (which are not a priori exhaustive)
depending on the parameters. 
Further, in the recurrent cases it is of interest to quantify the recurrence by studying
the tails (or moments) of return times to compact sets. This is the second main aim of this paper.

In the present paper we focus on the case where $\alpha^+ + \alpha^- =0$, which we call `opposed reflection'.
This
case is the most subtle from the point of view of recurrence/transience, and, as we will see, exhibits
a rich phase diagram depending on the model parameters.
We emphasize that the model in the case  $\alpha^+ + \alpha^- =0$ is \emph{near-critical}
in that both recurrence and transience are possible, depending on the parameters, and moreover (i) in the recurrent
cases, return-times to bounded sets have heavy tails being, in particular, non-integrable, and so stationary distributions
will not exist, and (ii) in the transient cases, escape to infinity will be only diffusive. 
There is a sense in which the model studied here can be viewed as a perturbation of zero-drift random walks,
in the manner of the seminal work of Lamperti~\cite{lamp3}:
see e.g.~\cite{mpw}
for a discussion of near-critical phenomena. We leave for future work the case  $\alpha^+ + \alpha^- \neq 0$, in which very different
behaviour will occur: if $\beta^\pm <1$, then the case $\alpha^+ + \alpha^- >0$ gives super-diffusive (but sub-ballistic) transience,
while the case $\alpha^+ + \alpha^- < 0$ leads to positive recurrence.

 Opposed reflection 
 includes the special case where $\alpha^+ = \alpha^- = 0$, which is `normal reflection'.
Since the results are in the latter case more easily digested, and since it is an important case
in its own right, we present the case of normal reflection first, in \S \ref{sec:normal}.
The general case of opposed reflection we present in \S \ref{sec:opposed}.
In \S \ref{sec:literature} we review some of the extensive related literature
on reflecting processes. 
Then \S \ref{sec:outline} gives an outline of the remainder of the paper,
which consists of the proofs of the results in \S\S \ref{sec:normal}--\ref{sec:opposed}.

\subsection{Normal reflection}
\label{sec:normal}

First we consider the case of \emph{normal} (i.e., orthogonal) reflection.

\begin{theorem}
\label{thm:normal-recurrence}
Suppose that~\eqref{ass:non-confinement}, \eqref{ass:bounded-moments}, \eqref{ass:zero-drift}, \eqref{ass:reflection}, and~\eqref{ass:covariance}
hold with $\alpha^+ = \alpha^- = 0$.
\begin{itemize}
\item[(a)] Suppose that $\beta^+, \beta^- \in [0,1)$. Let $\beta := \max ( \beta^+, \beta^-)$.
Then the following hold.
\begin{itemize}
\item[(i)] If $\beta < \sigma_1^2 / \sigma_2^2$, then $\xi$ is recurrent.
\item[(ii)] If $\sigma_1^2 / \sigma_2^2 < \beta < 1$, then $\xi$ is transient.
\item[(iii)] If, in addition,~\eqref{ass:zero-drift-plus} and~\eqref{ass:covariance-plus} hold, then the case $\beta = \sigma_1^2/\sigma_2^2$ is recurrent.
\end{itemize}
\item[(b)] Suppose that~\eqref{ass:zero-drift-plus} and~\eqref{ass:covariance-plus} hold, and $\beta^+,\beta^- >1$. Then $\xi$ is recurrent. 
\end{itemize}
\end{theorem}

\begin{remarks}
\begin{myenumi}
\setlength{\itemsep}{0pt plus 1pt}
\item[{\rm (i)}] Omitted from Theorem~\ref{thm:normal-recurrence} is the case when at least one of $\beta^\pm$ is equal to 1,
or their values fall each each side of~1. Here we anticipate behaviour similar to~\cite{aim}.

\item[{\rm (ii)}] If $\sigma_1^2 / \sigma_2^2   < 1$, then Theorem~\ref{thm:normal-recurrence} shows a striking \emph{non-monotonicity} property:
there exist regions $\cD_1 \subset \cD_2 \subset \cD_3$ such that the reflecting random walk is recurrent on $\cD_1$ and $\cD_3$, but transient
on $\cD_2$. This phenomenon does not occur in the classical case when $\Sigma$ is the identity: see~\cite{pinsky} for a derivation of monotonicity
in the case of normally reflecting Brownian motion in unbounded domains in $\R^d$, $d \geq 2$.

\item[{\rm (iii)}] Note that the correlation $\rho$ and the values of $a^+, a^-$ play no part in Theorem~\ref{thm:normal-recurrence};
$\rho$ will, however, play a role in the more general Theorem~\ref{thm:opposite} below.
\end{myenumi}
\end{remarks}
  
Let $\tau_r := \min \{ n \in \ZP : \| \xi_n \| \leq r \}$. Define
\begin{equation}
\label{eq:p-def}
 s_0 := s_0 (\Sigma,\beta) :=  \frac{1}{2} \left(1 - \frac{\sigma_2^2 \beta}{\sigma_1^2} \right) .\end{equation}
Our next result concerns the moments of $\tau_r$. Since most of our assumptions are
asymptotic, we only make statements about $r$ sufficiently large; with appropriate irreducibility assumptions,
this restriction could be removed.

\begin{theorem}
\label{thm:normal-moments}
Suppose that~\eqref{ass:non-confinement}, \eqref{ass:bounded-moments}, \eqref{ass:zero-drift}, \eqref{ass:reflection}, and~\eqref{ass:covariance}
hold with $\alpha^+ = \alpha^- = 0$.
\begin{itemize}
\item[(a)] Suppose that $\beta^+, \beta^- \in [0,1)$. Let $\beta := \max ( \beta^+, \beta^-)$.
Then the following hold.
\begin{itemize}
\item[(i)] If $\beta < \sigma_1^2 / \sigma_2^2$, then $\Exp_x ( \tau_r^s ) < \infty$
for all $s < s_0$ and all $r$ sufficiently large,
but  $\Exp_x ( \tau_r^s ) = \infty$
for all  $s > s_0$ and all $x$ with $\| x \| > r$ for $r$ sufficiently large.
\item[(ii)] If $\beta \geq \sigma_1^2 / \sigma_2^2$, then  $\Exp_x ( \tau_r^s ) = \infty$
for all  $s > 0$ and all $x$ with $\| x \| > r$  for $r$ sufficiently large.
\end{itemize}
\item[(b)] Suppose that $\beta^+,\beta^- >1$. Then $\Exp_x ( \tau_r^s ) = \infty$
for all  $s > 0$ and all $x$ with $\| x \| > r$  for $r$ sufficiently large. 
\end{itemize}
\end{theorem}

\begin{remarks}\phantomsection
\label{rem:moments}
\begin{myenumi}
\setlength{\itemsep}{0pt plus 1pt}
\item[{\rm (i)}] Note that
if $\beta <  \sigma_1^2 /\sigma_2^2$, then $s_0 > 0$, while
 $s_0 < 1/2$ for all $\beta >0$, in which case the return time to a bounded set has a heavier tail than that for one-dimensional
simple symmetric random walk. 

\item[{\rm (ii)}] The transience result in Theorem~\ref{thm:normal-recurrence}(a)(ii) 
is essentially stronger than the claim in Theorem~\ref{thm:normal-moments}(a)(ii) for $\beta < \sigma_1^2 / \sigma_2^2$, 
so the borderline (recurrent) case $\beta = \sigma_1^2 / \sigma_2^2$ is the main content of the latter.

\item[{\rm (iii)}] Part~(b) shows that the case $\beta^\pm >1$ is critical: no moments of return times exist,
as in the case of, say, simple symmetric random walk in $\Z^2$~\cite[p.~77]{mpw}.
\end{myenumi}
\end{remarks}

\subsection{Opposed reflection}
\label{sec:opposed}

We now consider the more general case where $\alpha^+ + \alpha^- =0$,
i.e., the two reflection angles are equal but opposite,
relative to their respective normal vectors. For $\alpha^+ = -\alpha^- \neq 0$,
this is a particular example of \emph{oblique} reflection.
The phase transition in $\beta$ now depends on $\rho$ and $\alpha$ in addition
to $\sigma_1^2$ and $\sigma_2^2$.
Define
\begin{equation}
\label{eq:betac}
\bc := \bc (\Sigma, \alpha ) := 
\frac{\sigma_1^2}{\sigma_2^2} + \left( \frac{\sigma_2^2-\sigma_1^2}{\sigma_2^2} \right) \sin^2 \alpha
+ \frac{\rho}{\sigma_2^2} \sin 2 \alpha 
.\end{equation}

The next result gives the key properties of the critical threshold function $\bc$
which are needed for interpreting our main result.

\begin{proposition}
\label{prop:bc-max-min}
For a fixed, positive-definite $\Sigma$
such that $|\sigma_1^2 - \sigma_2^2| + |\rho| > 0$, the function
$\alpha \mapsto \bc (\Sigma, \alpha )$ over the interval
$[ -\frac{\pi}{2}, \frac{\pi}{2}]$ is strictly
positive for $| \alpha | \leq \pi/2$, with
two stationary points, one in $(-\frac{\pi}{2},0)$ and
the other in $(0,\frac{\pi}{2})$, at which the function takes its maximum/minimum values of
\begin{equation}
\label{eq:bc-max-min}
 \frac{1}{2} + \frac{\sigma_1^2}{2\sigma_2^2} \pm \frac{1}{2\sigma_2^2} \sqrt{ \left( \sigma_1^2 - \sigma_2^2\right)^2 + 4 \rho^2 } .
\end{equation}
The exception is the case where $\sigma_1^2 - \sigma_2^2 = \rho = 0$, when $\bc = 1$ is constant.
\end{proposition}

Here is the recurrence classification in this setting.

\begin{theorem}
\label{thm:opposite}
Suppose that~\eqref{ass:non-confinement}, \eqref{ass:bounded-moments}, \eqref{ass:zero-drift}, \eqref{ass:reflection}, and~\eqref{ass:covariance}
hold with  $\alpha^+ = -\alpha^- = \alpha$ for $|\alpha| < \pi/2$.
\begin{itemize}
\item[(a)] Suppose that $\beta^+, \beta^- \in [0,1)$. Let $\beta := \max ( \beta^+, \beta^-)$.
Then the following hold.
\begin{itemize}
\item[(i)] If $\beta < \bc$, then $\xi$ is recurrent.
\item[(ii)] If $\beta > \bc$, then $\xi$ is transient.
\item[(iii)] If, in addition,~\eqref{ass:zero-drift-plus} and~\eqref{ass:covariance-plus} hold, then the case $\beta = \bc$ is recurrent.
\end{itemize}
\item[(b)] Suppose that~\eqref{ass:zero-drift-plus} and~\eqref{ass:covariance-plus} hold, and $\beta^+,\beta^- >1$. Then $\xi$ is recurrent. 
\end{itemize}
\end{theorem}
\begin{remarks}
\begin{myenumi}
\setlength{\itemsep}{0pt plus 1pt}
\item[{\rm (i)}]
The threshold~\eqref{eq:betac} is invariant under the map
$(\alpha,\rho) \mapsto (-\alpha,-\rho)$.

\item[{\rm (ii)}]
For fixed $\Sigma$ with $| \sigma_1^2 - \sigma_2^2 | + |\rho | >0$,
Proposition~\ref{prop:bc-max-min}
shows that $\bc$ is non-constant and has exactly one maximum and exactly one minimum in $(-\frac{\pi}{2}, \frac{\pi}{2})$.
Since $\bc (\Sigma, \pm \frac{\pi}{2} ) =1$, it follows from uniqueness of the minimum 
that the minimum is strictly less than $1$,
and so Theorem~\ref{thm:opposite} shows that there is always an open interval of $\alpha$ for which there is transience.

\item[{\rm (iii)}]
Since $\bc > 0$ always,  recurrence is certain for small enough $\beta$.

\item[{\rm (iv)}]
In the case where $\sigma_1^2 = \sigma_2^2$ and $\rho =0$, then $\bc =1$, so recurrence is certain for all $\beta^+, \beta^- <1$ and all $\alpha$.

\item[{\rm (v)}] If $\alpha = 0$, then $\bc = \sigma_1^2/\sigma_2^2$, so Theorem~\ref{thm:opposite}
generalizes Theorem~\ref{thm:normal-recurrence}.
\end{myenumi}
\end{remarks}

Next we turn to passage-time moments. We generalize~\eqref{eq:p-def} and define
\begin{equation}
\label{eq:p-def-opposite}
 s_0 := s_0 (\Sigma,\alpha,\beta) := \frac{1}{2} \left( 1 - \frac{\beta}{\bc} \right)
,\end{equation}
with $\bc$ given by~\eqref{eq:betac}. 
The next result includes Theorem~\ref{thm:normal-moments}
as the special case $\alpha =0$.

\begin{theorem}
\label{thm:opposite-moments}
Suppose that~\eqref{ass:non-confinement}, \eqref{ass:bounded-moments}, \eqref{ass:zero-drift}, \eqref{ass:reflection}, and~\eqref{ass:covariance}
hold with  $\alpha^+ = -\alpha^- = \alpha$ for $|\alpha| < \pi/2$.
\begin{itemize}
\item[(a)] Suppose that $\beta^+, \beta^- \in [0,1)$. Let $\beta := \max ( \beta^+, \beta^-)$.
Then the following hold.
\begin{itemize}
\item[(i)] If $\beta < \bc$, then $s_0 \in (0,1/2]$, and
 $\Exp_x ( \tau_r^s ) < \infty$
for all $s < s_0$ and all $r$ sufficiently large,
but  $\Exp_x ( \tau_r^s ) = \infty$
for all  $s > s_0$ and all $x$ with $\| x \| > r$  for $r$ sufficiently large.
\item[(ii)] If $\beta \geq \bc$, then  $\Exp_x ( \tau_r^s ) = \infty$
for all  $s > 0$ and all $x$ with $\| x \| > r$  for $r$ sufficiently large.
\end{itemize}
\item[(b)] Suppose that $\beta^+,\beta^- >1$. Then $\Exp_x ( \tau_r^s ) = \infty$
for all  $s > 0$ and all $x$ with $\| x \| > r$  for $r$ sufficiently large. 
\end{itemize}
\end{theorem}

\subsection{Related literature}
\label{sec:literature}

The stability properties of reflecting random walks or diffusions in unbounded domains in $\R^d$ have been studied for many years.
A pre-eminent place in the development of the theory is occupied by processes in the quadrant $\RP^2$
or quarter-lattice $\ZP^2$, due to applications arising in queueing theory and other areas. Typically, the process is assumed to be maximally homogeneous in the sense that the transition mechanism is
fixed in the interior and on each of the two half-lines making up the boundary.
Distinct are the cases where the motion in the interior of the domain has \emph{non-zero} or \emph{zero drift}.

It was in 1961, in part motivated by queueing models, that Kingman~\cite{kingman}
proposed a general approach to the non-zero drift problem on $\ZP^2$ via Lyapunov functions and Foster's Markov chain classification criteria~\cite{foster}.
A formal statement of the classification was given in the early 1970s by
Malyshev, who developed both an analytic approach~\cite{malyshev70} as well as the Lyapunov function one~\cite{malyshev72}
(the latter, Malyshev reports, prompted by a question of Kolmogorov). 
Generically, the classification depends on the drift vector in the interior and the two boundary reflection angles.
The Lyapunov function approach was further developed,
so that the bounded jumps condition in~\cite{malyshev72} could be relaxed to finiteness of second moments~\cite{mikhailov,fayolle,rosenkrantz} and, ultimately,
of first moments~\cite{fmm,vl,zachary}. The analytic approach was also subsequently developed~\cite{fim},
and although it seems to be not as robust as the Lyapunov function approach (the analysis in~\cite{malyshev70} was restricted to nearest-neighbour jumps), 
when it is applicable it can yield very precise
information: see e.g.~\cite{fr} for a recent application in the continuum setting. Intrinsically more complicated
 results are available for the non-zero drift case in $\ZP^3$~\cite{menshikov} and $\ZP^4$~\cite{im}.

The recurrence classification for the case of \emph{zero-drift}  reflecting random walk in $\ZP^2$ was given in the early 1990s in~\cite{afm,fmm92}; see also~\cite{fmm}.
In this case, generically, the classification depends on the increment covariance matrix in the interior as well as the two boundary reflection angles.
Subsequently,  using a semimartingale approach extending work of Lamperti~\cite{lamp3},
passage-time moments were studied in~\cite{aim},
with refinements provided in~\cite{aiSPA,aiTPA}. 

Parallel continuum developments 
concern reflecting Brownian motion in wedges in $\R^2$.
In the zero-drift case 
 with general (oblique) reflections, in the 1980s Varadhan and Williams~\cite{vw} had showed that the process was well-defined,
and then Williams~\cite{williams} gave the recurrence classification, thus preceding the random walk results of~\cite{afm,fmm92},
and, in the recurrent cases, asymptotics of stationary measures~(cf.~\cite{aiB} for the discrete setting). 
Passage-time moments were later studied in~\cite{mw,br}, by providing a continuum version of the results of~\cite{aim}, and in~\cite{aiSPA}, using
discrete approximation~\cite{aspandiiarov}.
The non-zero drift case was studied by Hobson and Rogers~\cite{hr}, who gave an analogue of Malyshev's theorem in the continuum setting.

For domains like our $\cD$, Pinsky~\cite{pinsky} established recurrence in the case of reflecting Brownian motion with
normal reflections and standard covariance matrix in the interior. The case of general covariance matrix and oblique reflection
does not appear to have been considered, and neither has the analysis of passage-time moments.
The somewhat related problem of the asymptotics of the \emph{first exit time} $\tau_e$ of planar Brownian motion from domains like our $\cD$ has been considered~\cite{bc,bds,li}: in the case where $\beta^+ = \beta^- = \beta \in (0,1)$,  then $\log \Pr ( \tau_e > t )$ is bounded above and below by constants times $-t^{(1-\beta)/(1+\beta)}$:
see~\cite{li} and (for the case $\beta =1/2$)~\cite{bds}.

\subsection{Overview of the proofs}
\label{sec:outline}

The basic strategy is to construct suitable Lyapunov functions $f : \R^2 \to \R$
that satisfy appropriate semimartingale (i.e., drift) conditions on $\Exp_x [ f (\xi_1 ) - f(\xi_0 ) ]$
for $x$ outside a bounded set. 
In fact,
since the Lyapunov functions that we use are most suitable for the case where the interior increment
covariance matrix is $\Sigma = I$, the identity, we first apply a linear transformation $T$ of $\R^2$
and work with $T \xi$. The linear transformation is described in~\S\ref{sec:linear}. 
Of course, one could combine these two steps and work directly with the Lyapunov function given by the composition
 $f \circ T$ for the appropriate $f$. However, for reasons of intuitive understanding and computational convenience, we prefer to separate the two steps.

Let $\beta^\pm < 1$. Then for $\alpha^+ = \alpha^- = 0$, the reflection angles are both pointing 
essentially vertically, with an asymptotically small component in the positive $x_1$ direction.
After the linear transformation~$T$, the reflection angles are no longer almost vertical, but instead are almost opposed
at some oblique angle, where the deviation from direct opposition is again asymptotically small, and in the positive $x_1$ direction.
For this reason, the case $\alpha^+ = -\alpha^- = \alpha \neq 0$ is not conceptually different
from the simpler case where $\alpha = 0$, because after the linear transformation, both cases are oblique.
In the case $\alpha \neq 0$, however, the details are more involved as both $\alpha$ and the value of the
correlation $\rho$ enter into the analysis of the Lyapunov functions, which is presented in~\S\ref{sec:lyapunov}, and is the main technical work of the paper.
For $\beta^\pm > 1$, intuition is provided by the case of reflection in the half-plane (see e.g.~\cite{williams} for the Brownian case).

Once the Lyapunov function estimates are in place, the proofs of the main theorems are given in \S\ref{sec:proofs},
using some semimartingale results which are variations on those from~\cite{mpw}. The appendix (\S\ref{sec:threshold})
contains the proof of Proposition~\ref{prop:bc-max-min} on the properties of the threshold function $\bc$ defined at~\eqref{eq:betac}.

\section{Linear transformation}
\label{sec:linear}

The inwards pointing normal vectors to $\partial \cD$ at $(x_1, d^\pm (x_1))$ are
\[ n^\pm (x_1) = \frac{1}{r^\pm(x_1)} \begin{pmatrix} a^\pm \beta^\pm x_1^{\beta^\pm -1} \\ \mp 1 \end{pmatrix},
\text{ where } r^\pm (x_1 ) := \sqrt{ 1 + (a^\pm )^2 (\beta^\pm)^2 x_1^{2 \beta^\pm -2} } .\]
Define 
\[ n^\pm_\perp (x_1) :=  \frac{1}{r^\pm(x_1)} \begin{pmatrix} \pm 1 \\ a^\pm \beta^\pm x_1^{\beta^\pm -1}  \end{pmatrix} .\]
Recall that $n^\pm (x_1, \alpha^\pm)$
is the unit vector at angle $\alpha^\pm$ to $n^\pm (x_1)$,
with positive angles measured anticlockwise (for $n^+$) or clockwise (for $n^-$).
Then (see Figure~\ref{fig:oblique} for the case of $n^+$) we have
$n^\pm (x_1, \alpha^\pm ) = n^\pm (x_1) \cos \alpha^\pm + n_\perp^\pm (x_1) \sin \alpha^\pm$,
so
\[ n^\pm ( x_1, \alpha^\pm ) = \frac{1}{r^\pm(x_1)} 
\begin{pmatrix}  \sin \alpha^\pm + a^\pm \beta^\pm x_1^{\beta^\pm -1} \cos \alpha^\pm
 \\ \mp \cos \alpha^\pm \pm a^\pm \beta^\pm x_1^{\beta^\pm -1} \sin \alpha^\pm  \end{pmatrix} .\]
In particular, if $\alpha ^+ = - \alpha^- = \alpha$,
\begin{equation}
\label{eq:oblique-vector}
 n^\pm ( x_1, \alpha^\pm ) = \frac{1}{r^\pm(x_1)} 
\begin{pmatrix} \pm \sin \alpha  + a^\pm \beta^\pm x_1^{\beta^\pm -1} \cos \alpha 
 \\ \mp \cos \alpha  + a^\pm \beta^\pm x_1^{\beta^\pm -1} \sin \alpha   \end{pmatrix} 
=: \begin{pmatrix} n_1^\pm ( x_1, \alpha^\pm ) 
 \\  n_2^\pm ( x_1, \alpha^\pm )   \end{pmatrix} .\end{equation}
Recall that $\Delta = \xi_1 - \xi_0$. Write $\Delta = (\Delta_1,\Delta_2)$ in components.

\begin{figure}
\begin{center}
\begin{tikzpicture}[domain=0:10, scale = 1.4,decoration={
    markings,
    mark=at position 0.6 with {\arrow[scale=2]{>}}}]
\draw[black, line width = 0.40mm]   plot[smooth,domain=1:7,samples=50] ({\x},  {(\x)^(1/2)});
\draw (4,2) -- (6,2.5);
\draw (2,1.5) -- (4,2);
\draw[postaction={decorate}] (4,2) -- (4.6,-0.4);
\draw (3.8,1.95) -- (3.86,1.71);
\draw (4.06,1.76) -- (3.86,1.71);
\draw (4.5,1.5) arc (-45:-67.5:1);
\draw[postaction={decorate}] (4,2) -- (6,0);
\draw[postaction={decorate}] (4.6,-0.4) -- (6,0);
\draw (4.7,-0.22) -- (4.56,-0.26);
\draw (4.7,-0.22) -- (4.74,-0.36);
\node at (4.48, 1.2)       {$\alpha^+$};
\node at (3.5,2.4)       {$(x_1, a^+x_1^{\beta^+})$};
\node at (3.3,0.5)       {$n^+(x_1) \cos \alpha^+$};
\node at (6.2,-0.4)       {$n_\perp^+(x_1) \sin \alpha^+$};
\node at (6,1)       {$n^+(x_1,\alpha^+)$};
\end{tikzpicture}
\end{center}
\caption{\label{fig:oblique} Diagram describing oblique reflection at angle $\alpha^+ >0$.}
\end{figure}
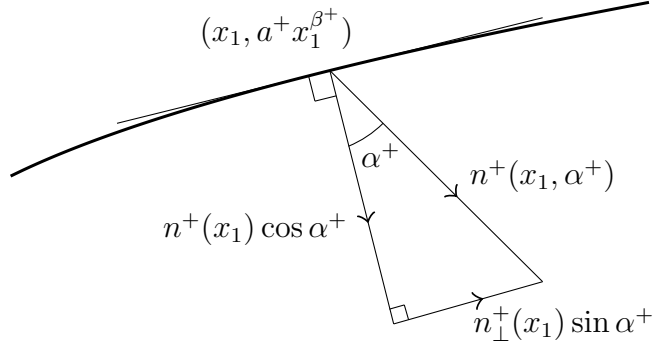

\begin{lemma}
\label{lem:increments}
Suppose that~\eqref{ass:reflection}
holds, with $\alpha^+  = -\alpha^- = \alpha$ and $\beta^+, \beta^- \geq 0$. 
If $\beta^\pm <1$, then, for $x \in S_B^\pm$, as $\| x \| \to \infty$,
\begin{align}
\label{eq:drift-boundary-small-beta1}
  \Exp_x   \Delta_1  & = \pm \mu^\pm (x) \sin \alpha + a^\pm \beta^\pm \mu^\pm (x) x_1^{\beta^\pm -1} \cos \alpha + O ( \| x\|^{2\beta^\pm -2})  + O ( \|x \|^{-1});\\
	\label{eq:drift-boundary-small-beta2}
	 \Exp_x   \Delta_2  & = \mp \mu^\pm (x)  \cos \alpha + a^\pm \beta^\pm \mu^\pm (x) x_1^{\beta^\pm -1} \sin \alpha   + O ( \| x\|^{2\beta^\pm -2})  + O ( \|x \|^{-1}) .\end{align}
If $\beta^\pm > 1$, then, for $x \in S_B^\pm$,  as $\| x \| \to \infty$,
\begin{align}
\label{eq:drift-boundary-big-beta1}
  \Exp_x   \Delta_1  & = \mu^\pm (x)  \cos \alpha \pm \frac{\mu^\pm (x)  \sin \alpha}{a^\pm \beta^\pm} x_1^{1-\beta^\pm}  + O (x_1^{2-2\beta^\pm}) + O ( \|x \|^{-1});\\
	\label{eq:drift-boundary-big-beta2}
	 \Exp_x   \Delta_2  & = \mu^\pm (x)  \sin \alpha \mp \frac{\mu^\pm (x) \cos \alpha}{a^\pm \beta^\pm} x_1^{1-\beta^\pm}   + O (x_1^{2-2\beta^\pm} )  + O ( \|x \|^{-1}).\end{align}
\end{lemma}
\begin{proof}
Suppose that $x \in S^\pm_B$.
By~\eqref{eq:reflection+}, we have that $\| \Exp_x \Delta -  \mu^\pm (x) n^\pm (x_1,\alpha^\pm)\| = O( \|x\|^{-1})$.
First suppose that $0 \leq \beta^\pm < 1$. 
Then, $1/r^\pm(x_1) = 1 + O (x_1^{2\beta^\pm-2})$, and hence,
by~\eqref{eq:oblique-vector},
\begin{align*}
 n_1^\pm (x_1,\alpha^\pm) & = \pm \sin \alpha + a^\pm \beta^\pm x_1^{\beta^\pm -1} \cos \alpha + O (x_1^{2\beta^\pm -2}) ;\\
 n_2^\pm (x_1,\alpha^\pm) & = \mp \cos \alpha + a^\pm \beta^\pm x_1^{\beta^\pm -1} \sin \alpha   + O (x_1^{2\beta^\pm -2} ).\end{align*}
Then, since $\| x \| = x_1 + o(x_1)$ as $\| x \| \to \infty$ with $x \in \cD$, 
we obtain~\eqref{eq:drift-boundary-small-beta1} and~\eqref{eq:drift-boundary-small-beta2}.

On the other hand, if $\beta^\pm > 1$, then 
\[ \frac{1}{r^\pm (x_1)} = \frac{x_1^{1-\beta^\pm}}{a^\pm \beta^\pm} + O (x_1^{3-3\beta^\pm}), \]
and hence,
by~\eqref{eq:oblique-vector},
\begin{align*}
 n_1^\pm (x_1,\alpha^\pm) & = \cos \alpha \pm \frac{\sin \alpha}{a^\pm \beta^\pm} x_1^{1-\beta^\pm}  + O (x_1^{2-2\beta^\pm}) ;\\
 n_2^\pm (x_1,\alpha^\pm) & = \sin \alpha \mp \frac{\cos \alpha}{a^\pm \beta^\pm} x_1^{1-\beta^\pm}   + O (x_1^{2-2\beta^\pm} ).\end{align*}
The expressions~\eqref{eq:drift-boundary-big-beta1} and~\eqref{eq:drift-boundary-big-beta2} follow.
\end{proof}

It is convenient to introduce a linear transformation of $\R^2$ under which the
asymptotic
increment covariance matrix $\Sigma$ appearing in~\eqref{ass:covariance} is transformed to the identity. 
Define
\[ 
T := \begin{pmatrix} \frac{\sigma_2}{s} & - \frac{\rho}{s\sigma_2} \\
0 & \frac{1}{\sigma_2} \end{pmatrix}, \text{ where } s := \sqrt{ \det \Sigma } = \sqrt{ \sigma_1^2 \sigma_2^2 - \rho^2 } ;\]
recall that $\sigma_2, s >0$, since $\Sigma$ is positive definite.
The choice of $T$ is such that $T \Sigma T^\tra = I$ (the identity), and 
$x \mapsto T x$ leaves the horizontal direction unchanged.
Explicitly,
\begin{equation}
\label{eq:T-transform}  T  \begin{pmatrix} x_1 \\ x_2 \end{pmatrix} = 
\begin{pmatrix}  \frac{\sigma_2}{s} x_1 - \frac{\rho}{s\sigma_2} x_2 \\
\frac{1}{\sigma_2} x_2 \end{pmatrix} .
\end{equation}
Note that $T$ is positive definite, and so $\| Tx\|$ is bounded above and below by positive constants
times $\| x \|$. Also, if $x \in \cD$ and $\beta^+,\beta^- < 1$, the fact that $| x_2 | = o(x_1)$ means that 
$Tx$ has the properties
(i) $(Tx)_1 >0$ for all $x_1$ sufficiently large, and (ii) 
 $| (Tx)_2| = o ( | (Tx)_1 | )$ as $x_1 \to \infty$. See Figure~\ref{fig:transform} for a picture.

\begin{figure}
\begin{center}
\scalebox{0.8}{
\begin{tikzpicture}[domain=0:10, scale = 1.0]
\filldraw (0,0) circle (1.5pt);
\draw[black,->,>=stealth,dashed] (0,0) -- (7,0);
\draw[black,->,>=stealth,dashed] (0,0) -- (-2,5);
\draw[black,->,>=stealth,dashed] (0,0) -- (2.5,1);
\draw[black,dashed] (0,0) -- (2,-5);
\draw[black, line width = 0.40mm]   plot[smooth,domain=0:1.0,samples=100] ({3*(\x)-2*(\x)^(2/3)},  {5*(\x)^(2/3)});
\draw[black, line width = 0.40mm]   plot[smooth,domain=0:1.0,samples=100] ({3*(\x)+2*(\x)^(2/3)},  {-5*(\x)^(2/3)});
\node at (-0.25,0) {$0$};
\node at (6.5,0.7) {$T \begin{pmatrix} x_1 \\ 0 \end{pmatrix}$};
\node at (6,-0.5) {$\theta = 0$};
\node at (3.0,1.2) {$\theta = \theta_2$};
\node at (-0.9,4.5) {$T \begin{pmatrix} 0 \\ x_2 \end{pmatrix}$};
\node at (1.8,4.5) {$T \partial \cD$};
\node at (-0.95,3) {$\theta = \theta_2 + \frac{\pi}{2}$};
\node at (0.4,-4) {$\theta = \theta_2 - \frac{\pi}{2}$};
\end{tikzpicture}}
\hfill
\scalebox{0.8}{
\begin{tikzpicture}[domain=0:10, scale = 1.0]
\filldraw (0,0) circle (1.5pt);
\draw[black,->,>=stealth,dashed] (0,0) -- (7,0);
\draw[black,->,>=stealth,dashed] (0,0) -- (-2,5);
\draw[black,->,>=stealth,dashed] (0,0) -- (2.5,1);
\draw[black,dashed] (0,0) -- (2,-5);
\draw[black, line width = 0.40mm]   plot[smooth,domain=0:1.0,samples=100] ({3*(\x)-2*(\x)^4},  {5*(\x)^4});
\draw[black, line width = 0.40mm]   plot[smooth,domain=0:1.0,samples=100] ({3*(\x)+2*(\x)^4},  {-5*(\x)^4});
\node at (-0.25,0) {$0$};
\node at (6.5,0.7) {$T \begin{pmatrix} x_1 \\ 0 \end{pmatrix}$};
\node at (6,-0.5) {$\theta = 0$};
\node at (3.0,1.2) {$\theta = \theta_2$};
\node at (-0.9,4.5) {$T \begin{pmatrix} 0 \\ x_2 \end{pmatrix}$};
\node at (1.8,4.5) {$T \partial \cD$};
\node at (0.0,3) {$\theta = \theta_2 + \frac{\pi}{2}$};
\node at (0.4,-4) {$\theta = \theta_2 - \frac{\pi}{2}$};
\end{tikzpicture}}
\caption{\label{fig:transform} An illustration of the transformation $T$ with $\rho >0$ acting on a domain $\cD$ with $\beta^+ = \beta^- = \beta$ for
$\beta \in (0,1)$ (\emph{left}) and $\beta >1$ (\emph{right}).
The angle $\theta_2$ is given by $\theta_2 = \arctan (\rho/s)$, measured anticlockwise from the positive horizontal axis.}
\end{center}
\end{figure}
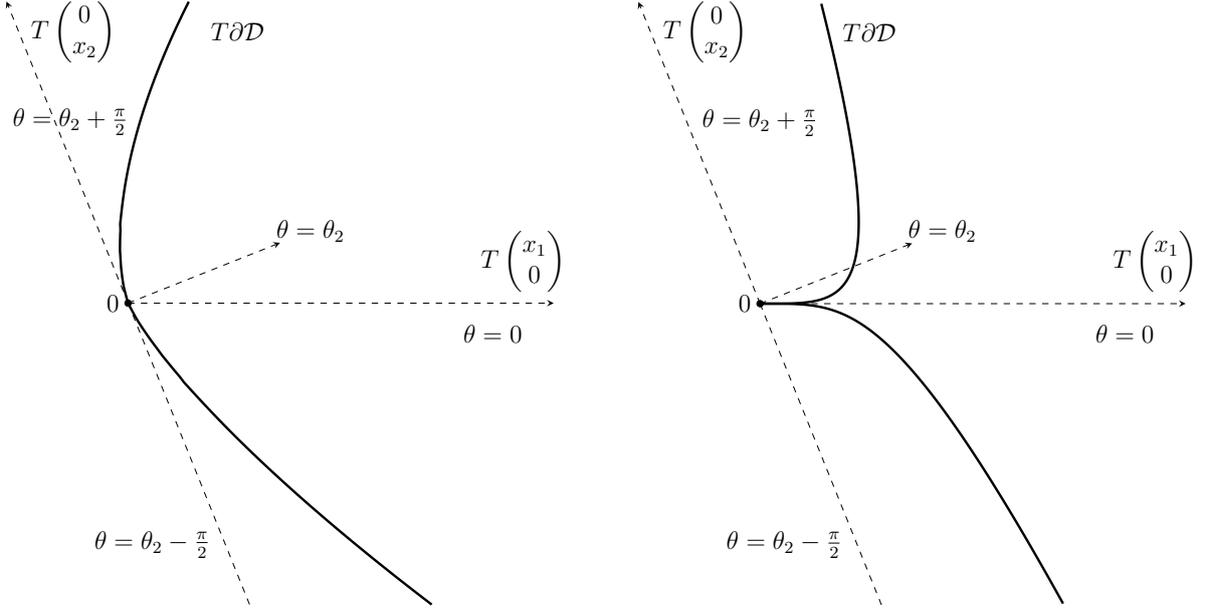 

The next result describes the increment moment properties of the process under the transformation $T$.
For convenience, we set $\tilde \Delta := T \Delta$ for the transformed increment, with components
$\tilde \Delta_i = ( T \Delta )_i$.

\begin{lemma}
\label{lem:transform}
Suppose that~\eqref{ass:zero-drift}, \eqref{ass:reflection}, and~\eqref{ass:covariance}
hold, with $\alpha^+  = -\alpha^-= \alpha$, and $\beta^+, \beta^- \geq 0$. 
Then, if $\| x \| \to \infty$ with $x \in S_I$, 
\begin{equation}
\label{eq:tilde-interior}
 \| \Exp_x    \tilde \Delta  \| = o ( \| x\|^{-1} ), 
\text{ and } \bigl\| \Exp_x (  \tilde \Delta  \tilde \Delta^\tra ) - I \bigr\|_{\mathrm{op}} = o(1) .\end{equation}
If, in addition, \eqref{ass:zero-drift-plus} and~\eqref{ass:covariance-plus} hold  with $\eps>0$, then, if $\| x \| \to \infty$ with $x \in S_I$, 
\begin{equation}
\label{eq:tilde-interior-plus}
 \| \Exp_x    \tilde \Delta  \| = O ( \| x\|^{-1-\eps} ), 
\text{ and } \bigl\| \Exp_x (  \tilde \Delta  \tilde \Delta^\tra ) - I \bigr\|_{\mathrm{op}} = O ( \| x\|^{-\eps} ).\end{equation}
If $\beta^\pm <1$, then, as $\| x \| \to \infty$ with $x \in S_B^\pm$,
\begin{align}
\label{eq:tilde-drift-boundary-small-beta1}
  \Exp_x   \tilde \Delta_1   & = \pm \frac{\sigma_2 \mu^\pm (x)}{s} \sin \alpha
	\pm 	\frac{\rho \mu^\pm(x)}{s\sigma_2} \cos \alpha
	+ \frac{\sigma_2 a^\pm \beta^\pm  \mu^\pm (x)}{s} x_1^{\beta^\pm-1} \cos \alpha \nonumber\\
	& {} \qquad \qquad {} 	- \frac{\rho a^\pm \beta^\pm  \mu^\pm (x)}{s \sigma_2} x_1^{\beta^\pm-1} \sin \alpha
	+ O (\|x\|^{2\beta^\pm -2} ) + O ( \| x \|^{-1} );\\
	\label{eq:tilde-drift-boundary-small-beta2}
	 \Exp_x   \tilde \Delta_2   & = \mp \frac{  \mu^\pm(x)}{\sigma_2} \cos \alpha
	+ \frac{a^\pm \beta^\pm  \mu^\pm (x)}{\sigma_2}  x_1^{\beta^\pm-1} \sin \alpha
	+   O (\|x\|^{2\beta^\pm -2} ) + O ( \| x \|^{-1} ).\end{align}
If $\beta^\pm > 1$, then,  as $\| x \| \to \infty$ with $x \in S_B^\pm$,
\begin{align}
\label{eq:tilde-drift-boundary-big-beta1}
 \Exp_x    \tilde \Delta_1   & =  \frac{\sigma_2  \mu^\pm (x)}{s} \cos \alpha - \frac{\rho \mu^\pm (x)}{s \sigma_2} \sin \alpha
 \pm \frac{\sigma_2 \mu^\pm (x)}{a^\pm \beta^\pm s} x_1^{1-\beta^\pm} \sin \alpha \nonumber\\
	& {} \qquad \qquad {}  \pm \frac{\rho \mu^\pm(x)}{a^\pm \beta^\pm s \sigma_2} x_1^{1-\beta^\pm} \cos \alpha + 
	O( x_1^{2-2\beta^\pm} ) + O ( \| x \|^{-1} ); \\
	\label{eq:tilde-drift-boundary-big-beta2}
	 \Exp_x   \tilde \Delta_2   & = 
	\frac{\mu^\pm (x)}{\sigma_2} \sin \alpha 
	\mp \frac{  \mu^\pm(x)}{a^\pm \beta^\pm \sigma_2} x_1^{1-\beta^\pm} \cos \alpha
	+ O( x_1^{2-2\beta^\pm} ) + O ( \| x \|^{-1} ).\end{align}
\end{lemma}
\begin{proof}
By linearity,
\begin{equation}
\label{eq:tilde-drift}
 \Exp_x   \tilde \Delta  = T  \Exp_x \Delta,
\end{equation}
which, by~\eqref{ass:zero-drift} or~\eqref{ass:zero-drift-plus}, is, respectively, $o(\|x\|^{-1})$ or~$O ( \| x \|^{-1-\eps})$ for $x \in S_I$. Also, 
since $T \Sigma T^\tra = I$, we have
\[ \Exp_x (   \tilde \Delta  \tilde \Delta^\tra ) - I
= T \Exp_x ( \Delta \Delta^\tra ) T ^\tra - I
= T \left( \Exp_x ( \Delta \Delta^\tra ) - \Sigma \right) T^\tra .\]
For $x \in S_I$, 
the middle matrix in the last product here has norm~$o(1)$ or~$O( \| x \|^{-\eps})$, by~\eqref{ass:covariance} or~\eqref{ass:covariance-plus}.
Thus we obtain~\eqref{eq:tilde-interior} and~\eqref{eq:tilde-interior-plus}.
For $x \in S^\pm_B$,
the claimed results follow on using~\eqref{eq:tilde-drift},~\eqref{eq:T-transform}, 
and the expressions for $\Exp_x \Delta$ in Lemma~\ref{lem:increments}.
\end{proof}

\section{Lyapunov functions}
\label{sec:lyapunov}

For the rest of the paper, we suppose that  $\alpha^+  = -\alpha^- = \alpha$
for some $| \alpha | < \pi/2$.
Our proofs will make use of some carefully chosen
functions of the process.
Most of these functions are most conveniently expressed in polar coordinates.

We write $x = (r, \theta)$ in polar coordinates, with angles measured relative to the 
positive horizontal axis: $r := r(x) := \| x \|$ and $\theta := \theta (x) \in (-\pi,\pi]$
is the angle between the ray through $0$ and $x$ and the ray in the Cartesian direction $(1,0)$,
with the convention that anticlockwise angles are positive. Then $x_1 = r \cos \theta$
and $x_2 = r \sin \theta$.

For $w \in \R$, $\theta_0 \in (-\pi/2,\pi/2)$, and $\gamma \in \R$,
 define 
\begin{equation}
\label{eq:h-w-def}
 h_w (x) := h_w (r,\theta) := r^w \cos (w\theta-\theta_0), ~\text{and}~ f^\gamma_w (x) := ( h_w (Tx) )^\gamma ,\end{equation}
where $T$ is the linear transformation describe at~\eqref{eq:T-transform}.
The functions $h_w$ were used in analysis of processes in wedges in e.g.~\cite{vw,rosenkrantz,aim,mmw}. Since
the $h_w$ are harmonic for the Laplacian (see below for a proof), Lemma~\ref{lem:transform} suggests that $h_w (T \xi_n)$
will be approximately a martingale in $S_I$, and the choice of the geometrical parameter~$\theta_0$ gives us the flexibility to
try to arrange things so that the level curves of $h_w$ are incident to the boundary at appropriate angles relative to the reflection vectors. The level curves of $h_w$ cross the horizontal axis at
angle $\theta_0$: see Figure~\ref{fig:level-curves}, and~\eqref{eq:h-w-derivatives} below.
In the case $\beta^\pm < 1$, the interest is near the horizontal axis,
and  we  take $\theta_0$ to be such that the level curves cut $\partial \cD$ at the reflection angles (asymptotically),
so that $h_w (T \xi_n)$ will be approximately a martingale also in $S_B$. Then adjusting $w$ and $\gamma$ will enable us
to obtain a supermartingale with the properties suitable to apply some Foster--Lyapunov theorems. 
This intuition is solidified in Lemma~\ref{lem:f-increments} below, where we show that 
the parameters $w$, $\theta_0$, and $\gamma$ can be chosen so that $f^\gamma_w (\xi_n)$
satisfies an appropriate supermartingale condition outside a bounded set. For the case $\beta^\pm < 1$, since
we only need to consider $\theta \approx 0$, we could replace these harmonic functions in polar coordinates by suitable polynomial approximations in Cartesian components, but since
we also want to consider $\beta^\pm > 1$, it is convenient to use the functions in the form given. When $\beta^\pm > 1$,
the recurrence classification is particularly delicate, so we must use another function (see~\eqref{eq:h-def} below),
although the functions at~\eqref{eq:h-w-def} will still be used to study passage time moments in that case.

\begin{figure}
\begin{center}
\includegraphics[width=0.5\textwidth]{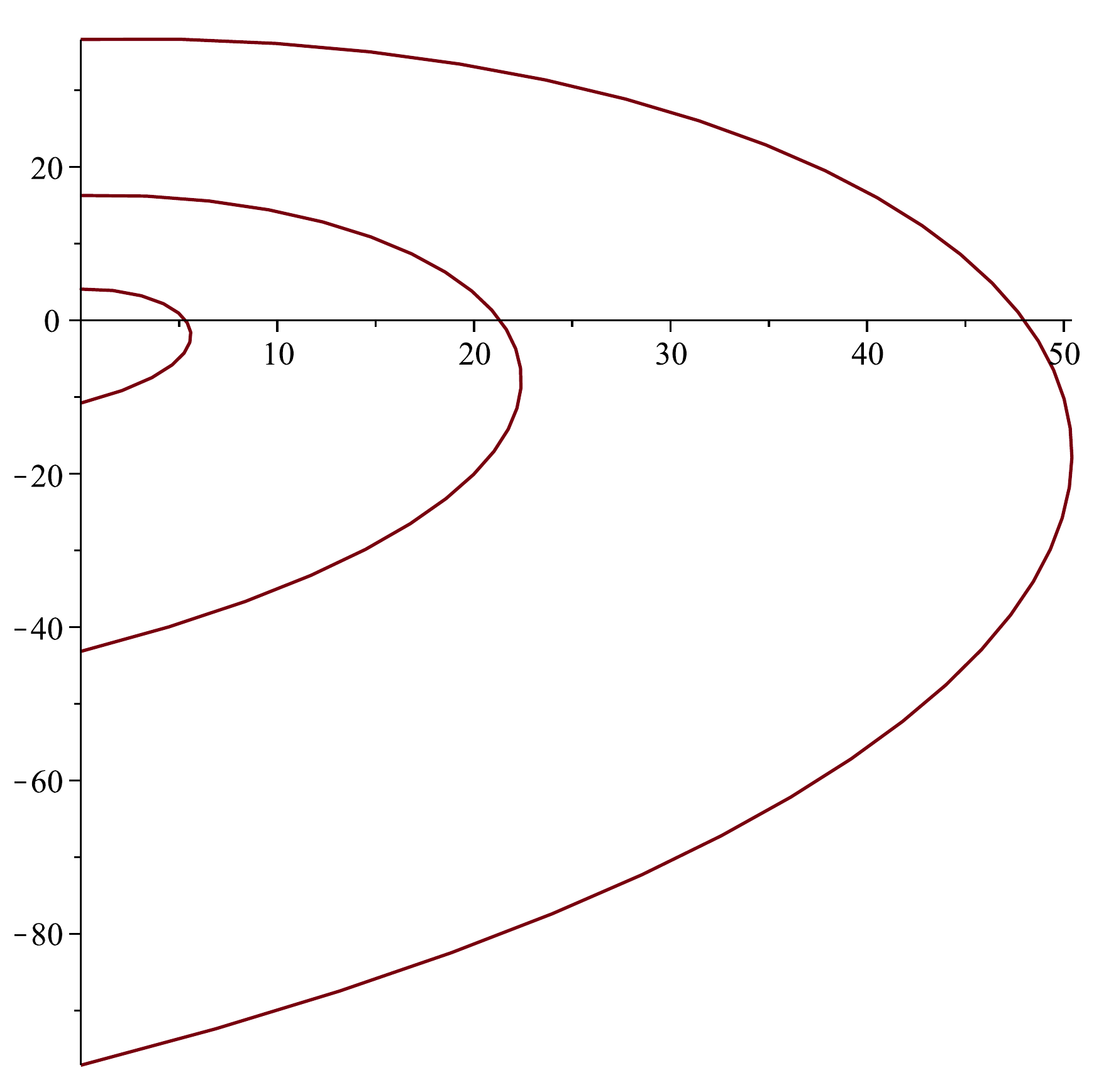}
\caption{\label{fig:level-curves} Level curves of the function $h_w (x)$ with $\theta_0 = \pi/6$ and $w = 1/4$. The level curves
cut the horizontal axis at angle~$\theta_0$ to the vertical.}
\end{center}
\end{figure}

If $\beta^+, \beta^- < 1$, then $\theta (x) \to 0$ as $\| x \| \to \infty$ with $x \in \cD$,
which means that, for any $| \theta_0 | < \pi/2$,
$h_w (x) \geq \delta \| x \|^w$ for some $\delta >0$ and all $x \in S$ with $\|x\|$ sufficiently large.
On the other hand, for $\beta^+, \beta^- >1$, we will restrict to the case with $w > 0$
sufficiently small such that $\cos ( w \theta - \theta_0 )$ is bounded away from zero, uniformly in $\theta \in [-\pi/2,\pi/2]$,
so that we again have the estimate $h_w (x) \geq \delta \| x \|^w$ for some $\delta >0$ and all $x \in \cD$, but where now
$\cD$ is close to the whole half-plane (see Remark~\ref{rem:w-choice}).
In the calculations that follow, we will often use the fact that $h_w (x)$ is bounded above and below by a constant 
times $\| x \|^w$ as $\| x \| \to \infty$ with $x \in \cD$.

We use the notation $D_i := \frac{\ud}{\ud x_i}$ for differentials, and for $f : \R^2 \to \R$ write $Df$ for the vector
with components $(Df)_i = D_i f$.
We use repeatedly
\begin{equation}
\label{eq:polar-diff}
 D_1 r = \cos \theta, ~ D_2 r = \sin \theta, ~ D_1 \theta  = - \frac{\sin \theta}{r}, ~ D_2 \theta = \frac{\cos \theta}{r} .\end{equation}
Define
\begin{equation}
\label{eq:theta-1}
\theta_1 := \theta_1 ( \Sigma, \alpha ) := \arctan \left( \frac{\sigma_2^2}{s} \tan \alpha + \frac{\rho}{s} \right) \in (-\pi/2, \pi/2) .
\end{equation}
For $\beta^\pm > 1$, we will also need 
\begin{equation}
\label{eq:theta-2}
\theta_2 := \theta_2 (\Sigma) := \arctan \left( \frac{\rho}{s} \right) \in (-\pi/2, \pi/2), \end{equation}
and $\theta_3 := \theta_3 ( \Sigma, \alpha ) \in (-\pi,\pi)$ for which 
\begin{equation}
\label{eq:sincos3}
\sin \theta_3 = \frac{s \sin \alpha}{\sigma_2 d}, \text{ and } \cos \theta_3 = \frac{\sigma_2^2 \cos \alpha - \rho \sin \alpha}{\sigma_2 d} ,\end{equation}
where 
\begin{equation}
\label{eq:d-def}
d := d (\Sigma,\alpha) := \sqrt{ \sigma_2^2 \cos^2 \alpha - 2 \rho \sin \alpha \cos \alpha + \sigma_1^2 \sin^2 \alpha } .\end{equation}
The geometric interpretation of $\theta_1, \theta_2$, and $\theta_3$ is as follows.
\begin{itemize}
\item The angle between $(0,\pm 1)$ and $T (0,\pm 1)$ has magnitude $\theta_2$. Thus, if $\beta^\pm < 1$, then $\theta_2$ is, as $x_1 \to \infty$,
the limiting angle of the transformed inwards pointing normal at $x_1$ relative to the vertical.
On the other hand, if $\beta^\pm > 1$, then $\theta_2$ is, as $x_1 \to \infty$,
the limiting angle, relative to the horizontal, of the
inwards pointing normal to $T \partial \cD$.
See Figure~\ref{fig:transform}.
\item The angle between $(0,-1)$ and $T (\sin \alpha, -\cos \alpha)$ is $\theta_1$.
Thus, if $\beta^\pm < 1$, then $\theta_1$ is, as $x_1 \to \infty$,
the limiting angle between the vertical and the transformed reflection vector. Since the normal
in the transformed domain remains asymptotically vertical, $\theta_1$ is in this case the limiting reflection
angle, relative to the normal, after the transformation.
\item The angle between $(1,0)$ and $T ( \cos \alpha, \sin \alpha)$ is $\theta_3$.
Thus, if $\beta^\pm > 1$, then $\theta_3$ is, as $x_1 \to \infty$,
the limiting angle between the horizontal  and the transformed reflection vector.
Since the transformed normal is, asymptotically, at angle $\theta_2$ relative to the horizontal,
 the limiting reflection
angle, relative to the normal, after the transformation is in this case $\theta_3 - \theta_2$.
\end{itemize}
We need two simple facts.

\begin{lemma}
\label{lem:facts}
We have 
(i)~$\inf_{\alpha \in [-\frac{\pi}{2},\frac{\pi}{2}]} d (\Sigma,\alpha) >0$, and~(ii)
$| \theta_3 - \theta_2 | < \pi/2$.
\end{lemma}
\begin{proof}
For~(i), from~\eqref{eq:d-def} we may write
\begin{equation}
\label{eq:d-squared} d^2 = \sigma_2^2 + \left( \sigma_1^2 - \sigma_2^2 \right) \sin^2 \alpha  -  \rho \sin 2\alpha .\end{equation}
If $\sigma_1^2 \neq \sigma_2^2$, then, by Lemma~\ref{lem:calculus}, the extrema over $\alpha \in [-\frac{\pi}{2},\frac{\pi}{2}]$ of~\eqref{eq:d-squared} are
\begin{align*} \sigma_2^2 + \frac{\sigma_1^2 - \sigma_2^2 }{2} \left(1 \pm \sqrt{1 + \frac{4\rho^2}{(\sigma_1^2 - \sigma_2^2)^2}} \right). \end{align*}
Hence
\[ d^2 \geq \frac{\sigma_1^2 +\sigma_2^2}{2} - \frac{1}{2} \sqrt{ (\sigma_1^2 - \sigma_2^2)^2 + 4 \rho^2 } ,\]
which is strictly positive since $\rho^2 < \sigma_1^2 \sigma_2^2$. If $\sigma_1^2 = \sigma_2^2$, then
$d^2 \geq \sigma_2^2 - | \rho |$, and $| \rho | < | \sigma_1 \sigma_2 | = \sigma_2^2$, so $d$ is also strictly positive in that case.

For~(ii), we use the fact that
$\cos (\theta_3 - \theta_2 ) = \cos \theta_3 \cos \theta_2 + \sin \theta_3 \sin \theta_2$,
where, by~\eqref{eq:theta-2}, $\sin \theta_2 = \frac{\rho}{\sigma_1\sigma_2}$
and $\cos \theta_2 = \frac{s}{\sigma_1\sigma_2}$, and~\eqref{eq:sincos3},
to get
$\cos (\theta_3 - \theta_2 ) = \frac{s}{\sigma_1 d} \cos \alpha > 0$. Since $| \theta_3 - \theta_2 | < 3\pi/2$,
it follows that $|\theta_3 - \theta_2| < \pi/2$, as claimed.
\end{proof}

We estimate the expected increments of our Lyapunov functions in two stages:
the main term comes from a Taylor expansion valid when the jump of the walk is not too big
compared to its current distance from the origin, while we bound the (smaller) contribution from big jumps
using the moments assumption~\eqref{ass:bounded-moments}. For the first stage,
let $B_b (x) := \{ z \in \R^2 : \| x - z \| \leq b \}$
denote the (closed) Euclidean ball centred at $x$ with radius $b \geq 0$.
We use the multivariable Taylor theorem in the following form.
Suppose that $f : \R^2 \to \R$ is thrice continuously differentiable
in $B_b (x)$. Recall that  $Df(x)$ is the vector
function whose components are $D_i f(x)$. 
Then, for $y \in B_b (x)$,
\begin{equation}
\label{eq:taylor-second-order}
f (x + y) = f(x) + \langle D f(x) , y \rangle + y_1^2 \frac{D^2_{1} f(x)}{2}   + y_2^2 \frac{D^2_{2} f(x)}{2}  
+ y_1 y_2 D_1 D_2 f(x) + R (x,y) ,\end{equation}
where, for all $y \in B_b(x)$, $| R (x,y ) | \leq C \| y \|^3 R  (x)$ for an absolute constant $C < \infty$ and 
\[ R  (x) :=  \max_{i,j,k} \sup_{ z \in B_b(x) } \left| D_i D_j D_k f (z) \right| .\]
For dealing with the large jumps, we observe the useful fact that if $p>2$
is a constant for which~\eqref{eq:p-moments} holds, then for some constant $C< \infty$, all $\delta \in (0,1)$,
 and all $q \in [0,p]$,
\begin{equation}
\label{eq:big-jumps}
\Exp_x \bigl[ \| \Delta \|^q \1  { \| \Delta \| \geq \| x \|^\delta  } \bigr] \leq C \| x \|^{- \delta(p-q)} ,\end{equation}
for all $\| x\|$ sufficiently large. To see~\eqref{eq:big-jumps}, write $\| \Delta \|^q = \| \Delta \|^p \| \Delta \|^{q-p}$
and use the fact that $\| \Delta \| \geq \| x \|^\delta$ to bound the second factor.

Here is our first main Lyapunov function estimate.

\begin{lemma}
\label{lem:f-increments}
Suppose that~\eqref{ass:bounded-moments}, \eqref{ass:zero-drift}, \eqref{ass:reflection}, and~\eqref{ass:covariance}
hold, with $p>2$,  
$\alpha^+  = -\alpha^- = \alpha$ for $| \alpha | < \pi/2$,
and $\beta^+, \beta^- \geq 0$.
Let $w, \gamma \in \R$ be such that
$2-p < \gamma w < p$. Take
$\theta_0 \in (-\pi/2,\pi/2)$.
Then as $\| x \| \to \infty$ with $x \in S_I$,
\begin{align}
\label{eq:f-increment-interior}
 \Exp [   f^\gamma_w(\xi_{n+1})  - f^\gamma_w( \xi_n) \mid \xi_n = x  ] 
& = \frac{\gamma(\gamma-1)}{2} w^2 ( h_w(Tx) )^{\gamma-2} \| Tx \|^{2w-2}  + o ( \|x\|^{\gamma w -2}).
\end{align}
We separate the boundary behaviour into two cases.
\begin{itemize}
\item[(i)] If $0 \leq \beta^\pm < 1$, take $\theta_0 = \theta_1$ given by~\eqref{eq:theta-1}.
  Then,
as $\| x \| \to \infty$ with $x \in S^\pm_B$,
\begin{align}
\label{eq:f-increment-boundary-small-beta}
& {} \Exp [  f^\gamma_w(\xi_{n+1})  - f^\gamma_w( \xi_n) \mid \xi_n = x ]  \\
& {} \quad {} = \gamma  w \| Tx \|^{w-1} \left( h_w (Tx ) \right)^{\gamma-1}
  \frac{a^\pm \mu^\pm (x) \sigma_2 \cos \theta_1 }{s \cos \alpha} \left(\beta^\pm - (1-w) \bc  +o(1) \right)
x_1^{\beta^\pm-1}
, \nonumber \end{align}
where $\bc$ is given by~\eqref{eq:betac}.
\item[(ii)] If $\beta^\pm > 1$, suppose that $w \in (0,1/2)$ and 
 $\theta_0 = \theta_0 (\Sigma,\alpha,w) = \theta_3 - (1-w)\theta_2$,
where $\theta_2$ and $\theta_3$ are given by~\eqref{eq:theta-2} and~\eqref{eq:sincos3},
such that $\sup_{\theta \in [-\frac{\pi}{2},\frac{\pi}{2}]} |w \theta - \theta_0 | < \pi/2$.
Then, with $d = d(\Sigma,\alpha)$ as defined at~\eqref{eq:d-def},
as $\| x \| \to \infty$ with $x \in S^\pm_B$,
\begin{align}
\label{eq:f-increment-boundary-big-beta}
& {} \Exp [  f^\gamma_w(\xi_{n+1})  - f^\gamma_w( \xi_n) \mid \xi_n = x ] \nonumber\\
& {} \quad {}= \gamma w \| Tx \|^{w-1} \left( h_w (Tx ) \right)^{\gamma-1}
  \frac{d \mu^\pm (x)}{s}  \left( \cos ( (1-w) (\pi/2) )   + o(1) \right)
 .\end{align}
\end{itemize}
\end{lemma}
\begin{remark}
\label{rem:w-choice}
We can choose $w>0$ small enough so that $| \theta_3 - (1-w)\theta_2 | < \pi/2$, 
by Lemma~\ref{lem:facts}(ii), and so if $\theta_0 = \theta_3 - (1-w)\theta_2$,
we can always choose $w>0$ small enough so that 
$\sup_{\theta \in [-\frac{\pi}{2},\frac{\pi}{2}]} |w \theta - \theta_0 | < \pi/2$,
as required for the $\beta^\pm > 1$ part of Lemma~\ref{lem:f-increments}.
\end{remark}
\begin{proof}[Proof of Lemma~\ref{lem:f-increments}.]
Differentiating~\eqref{eq:h-w-def} and using~\eqref{eq:polar-diff} we see that
\begin{align}
 D_1 h_w(x) & = w r^{w-1} \cos \left( (w-1) \theta - \theta_0 \right) , \text{ and } \nonumber\\
 D_2 h_w (x) & =  - w r^{w-1} \sin \left( (w-1) \theta - \theta_0 \right)
\label{eq:h-w-derivatives}  .\end{align}
Moreover,
\[ D_1^2 h_w(x) =  w (w-1) r^{w-2} \cos \left( (w-2)\theta - \theta_0 \right) = - D_2^2 h_w(x) ,\]
verifying that $h_w$ is harmonic. 
Also, for any $i,j,k$, $| D_i D_j D_k h_w(x) | = O (r^{w-3} )$.
Writing $h_w^\gamma (x) := ( h_w (x) )^\gamma$, we also have that
$D_i h_w^\gamma (x) = \gamma h_w^{\gamma-1} (x) D_i h_w (x)$, that
\begin{align*}
 D_i D_j h_w^\gamma (x) & = \gamma h_w^{\gamma-1} (x) D_i D_j h_w (x) + \gamma (\gamma -1) h_w^{\gamma -2} (x) ( D_i h_w (x) ) (D_j h_w(x)) ,\end{align*}
and $| D_i D_j D_k h_w^\gamma (x) | = O ( r^{\gamma w -3} )$.
We apply Taylor's formula~\eqref{eq:taylor-second-order} in the ball $B_{r/2}(x)$
together with the harmonic property of $h_w$, to obtain, for $y \in B_{r/2}(x)$,
\begin{align}
\label{eq:h-w-taylor}
 h^\gamma_w (x +y )  
& = 
  h^\gamma_w (x )
+  \gamma \langle D h_w(x), y \rangle  h_w ^{\gamma -1} (x) 
+ \frac{\gamma (\gamma -1)}{2} \langle D h_w(x), y \rangle^2    h^{\gamma-2}_w (x )  \nonumber\\
& {} \qquad {} +
\gamma \left(
\frac{(y_1^2 - y_2^2) D_1^2 h_w(x)}{2}  
+   y_1 y_2 D_{1} D_{2} h_w(x) \right) h^{\gamma-1}_w (x ) + R(x,y) , 
\end{align}
where $| R(x,y) | \leq C \| y \|^3 \| x \|^{\gamma w -3 }$, using the fact that $h_w (x)$ is bounded
above and below by a constant times $\| x\|^w$.

Let $E_x := \{ \|  \Delta \| < \| x \|^\delta \}$, where throughout the proof we fix a constant $\delta$ satisfying
\begin{equation}
\label{eq:delta-choice}
  \frac{\max \{ 2 , \gamma w , 2 -\gamma w \}}{p} < \delta < 1 ; \end{equation}
 such a choice of $\delta$ is possible since $p >2$ and $ 2 - p < \gamma w < p$.
If $\xi_0 = x$ and $E_x$ occurs, then $Tx + \tilde \Delta \in B_{r/2} (Tx)$ for all $\|x\|$ sufficiently large.
Thus, conditioning on $\xi_0 = x$, on the event $E_x$ we may use  
  the expansion in~\eqref{eq:h-w-taylor} for $h^\gamma_w (Tx + \tilde \Delta )$, which, after taking expectations,
	yields
\begin{align}
\label{eq:f-gamma-big-display}
& {} \Exp_x \bigl[ ( f_w^\gamma (\xi_1) -f_w^\gamma (\xi_0) )\2 {E_x} \bigr]
 = \gamma  \left( h_w (Tx ) \right)^{\gamma-1}  \Exp_x \bigl[ \langle D h_w(Tx), \tilde \Delta \rangle \2 {E_x} \bigr] 
\nonumber\\
& {} \qquad {}  +
 \gamma \left( h_w (Tx ) \right)^{\gamma-1} \left[
\frac{ D_1^2 h_w(Tx) \Exp_x \bigl[ ( \tilde \Delta_1^2 - \tilde \Delta_2^2) \2 {E_x} \bigr]}{2} + 
D_{1} D_{2} h_w(Tx) \Exp_x  \bigl[ \tilde \Delta_1 \tilde \Delta_2  \2 {E_x} \bigr]    \right]
\nonumber\\
& {} \qquad {} + \frac{\gamma (\gamma -1)}{2}  \left( h_w (Tx ) \right)^{\gamma-2} \Exp_x \bigl[ \langle D h_w(Tx), \tilde \Delta \rangle^2   \2 {E_x} \bigr] 
 + \Exp_x \bigl[ R(Tx,\tilde \Delta) \2 {E_x} \bigr]
.\end{align}
Let $p' = p \wedge 3$, so that~\eqref{eq:p-moments} also holds for $p' \in (2,3]$.
Then, writing $\| \tilde \Delta \|^3 =  \| \tilde \Delta \|^{p'}  \| \tilde \Delta \|^{3-p'}$,
\[  \Exp_x \bigl[ | R(Tx,\tilde \Delta) | \2 {E_x} \bigr] \leq C \| x \|^{\gamma w -3 + (3-p')\delta}  \Exp_x \bigl[ \| \tilde \Delta \|^{p'}  \bigr]
= o ( \| x \|^{\gamma w -2} ) ,\]
since $(3-p')\delta < 1$.
If $x \in S_I$, then~\eqref{eq:tilde-interior} shows that
 $| \Exp_x \langle D h_w(Tx), \tilde \Delta \rangle | = o ( \| x \|^{w-2} )$, so
\[
  \Exp_x \bigl| \langle D h_w(Tx), \tilde \Delta \rangle   \2 {E_x} \bigr|
\leq C \| x \|^{w-1} \Exp_x ( \| \Delta \|  \2 {E^\rc_x} ) + o ( \| x \|^{w-2} ) .\]
Note that, by~\eqref{eq:delta-choice}, $\delta > \frac{2}{p} > \frac{1}{p-1}$.
Then, using the $q=1$ case of~\eqref{eq:big-jumps}, we get
\begin{equation}
\label{eq:drift-small-jump}
\Exp_x \bigl| \langle D h_w(Tx), \tilde \Delta \rangle    \2 {E_x} \bigr| = o ( \| x \|^{w-2} ) .\end{equation}
A similar argument using the $q=2$ case of~\eqref{eq:big-jumps} gives
\[ \Exp_x \bigl[ \langle D h_w(Tx), \tilde \Delta \rangle^2   \2 {E^\rc_x} \bigr] \leq C \| x \|^{2w-2-\delta(p-2)} = o ( \| x \|^{2w-2} ).\] 
If  $x \in S_I$, then~\eqref{eq:tilde-interior}   shows that
$\Exp_x (\tilde \Delta_1^2 - \tilde \Delta_2^2)$ and $\Exp_x ( \tilde \Delta_1 \tilde \Delta_2 )$ are both $o(1)$, and
hence, by the~$q=2$ case of~\eqref{eq:big-jumps} once more, we see that $\Exp_x  [ | \tilde \Delta_1^2 - \tilde \Delta_2^2| \2 {E_x} ]$
and $\Exp_x [| \tilde \Delta_1 \tilde \Delta_2 | \2 {E_x} ]$ are both $o(1)$.
Moreover,~\eqref{eq:tilde-interior} also  shows that
\begin{align*} \Exp_x \langle D h_w(Tx), \tilde \Delta \rangle^2 & = \Exp_x \left( ( Dh_w (Tx) )^\tra \tilde \Delta \tilde \Delta^\tra Dh_w (Tx) \right)\\
& = ( Dh_w (Tx) )^\tra  Dh_w (Tx) + o ( \| x \|^{2w-2} ) \\
& = ( D_1 h_w(Tx) )^2 + ( D_2 h_w(Tx) )^2 + o ( \| x \|^{2w-2} ) .\end{align*}
Putting all these estimates into~\eqref{eq:f-gamma-big-display} we get, for $x \in S_I$,
\begin{align}
\label{eq:f-gamma-small-jump}
\Exp_x \bigl[ ( f_w^\gamma (\xi_1) -f_w^\gamma (\xi_0) )\2 {E_x} \bigr]
& = \frac{\gamma (\gamma -1)}{2}  \left( h_w (Tx ) \right)^{\gamma-2} 
\left( ( D_1 h_w(Tx) )^2 + ( D_2 h_w(Tx) )^2 \right) \nonumber\\
& {} \qquad {}   + o ( \| x \|^{\gamma w - 2} ) .\end{align}
On the other hand, given $\xi_0 = x$, if $\gamma w \geq 0$, by the triangle inequality,
\begin{align}
\label{eq:f-crude-bound}
\bigl| f_w^\gamma (\xi_1) -f_w^\gamma (x) \bigr| 
& \leq \| T\xi_1 \|^{\gamma w} + \| Tx \|^{\gamma w}
\leq 2 \bigl( \| T\xi_1 \| + \| Tx \| \bigr)^{\gamma w} \nonumber\\
& \leq 2 \bigl( 2 \| Tx \| + \| \tilde \Delta \| \bigr)^{\gamma w} .\end{align}
It follows from~\eqref{eq:f-crude-bound} that $| f_w^\gamma (\xi_1) -f_w^\gamma (x) | \2 {E_x^\rc } \leq C \| \Delta \|^{\gamma w /\delta}$, 
for some constant $C < \infty$ and all $\| x \|$ sufficiently large. Hence
\[ \Exp_x \bigl| ( f_w^\gamma (\xi_1) -f_w^\gamma (\xi_0) )\2 {E^\rc_x} \bigr|
\leq C \Exp_x \bigl[ \| \Delta \|^{\gamma w /\delta} \2 {E^\rc_x} \bigr] .\]
Since $\delta > \frac{\gamma w}{p}$, by~\eqref{eq:delta-choice}, we may apply~\eqref{eq:big-jumps}
with $q = \frac{\gamma w}{\delta}$ to get
\begin{align}
\label{eq:f-gamma-big-jump} \Exp_x \bigl| ( f_w^\gamma (\xi_1) -f_w^\gamma (\xi_0) )\2 {E^\rc_x} \bigr|
= O ( \| x \|^{\gamma w - \delta p } ) = o ( \| x \|^{\gamma w - 2} ),\end{align}
since $\delta > \frac{2}{p}$.
If $w \gamma < 0$, then we simply use the fact that $f_w^\gamma$ is uniformly bounded to get
\[ \Exp_x \bigl| ( f_w^\gamma (\xi_1) -f_w^\gamma (\xi_0) )\2 {E^\rc_x} \bigr|
\leq C \Pr_x  ( E^\rc_x )
= O ( \| x \|^{-\delta p} ) ,\]
by the $q=0$ case of~\eqref{eq:big-jumps}. Thus~\eqref{eq:f-gamma-big-jump}
holds in this case too, since $\gamma w > 2 -\delta p$ by choice of $\delta$ at~\eqref{eq:delta-choice}.
Then~\eqref{eq:f-increment-interior} follows from
combining~\eqref{eq:f-gamma-small-jump}
and~\eqref{eq:f-gamma-big-jump} with~\eqref{eq:h-w-derivatives}.

Next suppose that $x \in S_B$. Truncating~\eqref{eq:h-w-taylor}, we see that for all $y \in B_{r/2} (x)$,
\begin{equation}
\label{eq:h-w-taylor-truncated}
 h^\gamma_w (x +y )  
  = 
  h^\gamma_w (x )
+  \gamma \langle D h_w(x), y \rangle  h_w ^{\gamma -1} (x) 
+ R(x,y) ,\end{equation}
where now $|R(x,y) | \leq C \| y \|^2 \| x \|^{\gamma w -2 }$.
It follows from~\eqref{eq:h-w-taylor-truncated} and~\eqref{ass:bounded-moments} that
\begin{align*}
\Exp_x \bigl[ ( f_w^\gamma (\xi_1) -f_w^\gamma (\xi_0) ) \2 {E_x} \bigr]
& = \gamma  h^{\gamma -1}_w (Tx )   \Exp_x \bigl[ \langle D h_w(Tx), \tilde \Delta \rangle \2 {E_x} \bigr] 
+ O ( \| x \|^{\gamma w -2} ).
\end{align*}
By the $q=1$ case of~\eqref{eq:big-jumps}, since $\delta > \frac{1}{p-1}$, we see 
that $\Exp_x [ \langle D h_w(Tx), \tilde \Delta \rangle \2 {E^\rc_x} ] = o ( \| x \|^{w-2} )$, 
while the estimate~\eqref{eq:f-gamma-big-jump} still applies, so that
\begin{align}
\label{eq:f-boundary}
\Exp_x \bigl[   f_w^\gamma (\xi_1) -f_w^\gamma (\xi_0)   \bigr]
& = \gamma  h^{\gamma -1}_w (Tx )   \Exp_x  \langle D h_w(Tx), \tilde \Delta \rangle  
+ O ( \| x \|^{\gamma w -2} ).
\end{align}
From~\eqref{eq:h-w-derivatives} we have
\begin{equation}
\label{eq:h-w-D-vector}
 D h_w (Tx) = w \| Tx\|^{w-1} \begin{pmatrix} \cos ( (1-w)\theta(Tx) + \theta_0 ) \\
 \sin ( (1-w)\theta(Tx) + \theta_0 ) \end{pmatrix} .\end{equation}
First suppose that $\beta^\pm < 1$. Then, by~\eqref{eq:T-transform}, for $x \in S_B^\pm$, $x_2 = \pm a^\pm x_1^{\beta^\pm} + O(1)$ and
\[ \sin \theta (Tx) = \pm \frac{s a^\pm}{\sigma_2^2} x_1^{\beta^\pm -1} + O (x_1^{2\beta^\pm -2} ) + O( x_1^{-1} ) .\]
Since $\arcsin z = z + O (z^3)$ as $z \to 0$, it follows that
\[ \theta (Tx) = \pm \frac{s a^\pm}{\sigma_2^2} x_1^{\beta^\pm -1} + O (x_1^{2\beta^\pm -2} ) + O( x_1^{-1} ) .\]
Hence
\begin{align*}
 \cos \left( (1-w) \theta (Tx) + \theta_0 \right) & = \cos \theta_0 \mp (1-w) \frac{s a^\pm}{\sigma_2^2} x_1^{\beta^\pm-1} \sin \theta_0 + O (x_1^{2\beta^\pm-2} ) + O( x_1^{-1} ) ; \\
 \sin \left( (1-w) \theta (Tx) + \theta_0 \right) & =  \sin \theta_0 \pm (1-w) \frac{s a^\pm}{\sigma_2^2} x_1^{\beta^\pm-1} \cos \theta_0 + O (x_1^{2\beta^\pm-2} ) + O( x_1^{-1} ) .\end{align*}
Then~\eqref{eq:h-w-D-vector} with~\eqref{eq:tilde-drift-boundary-small-beta1} and~\eqref{eq:tilde-drift-boundary-small-beta2} shows that 
\begin{align}
\label{eq:h-w-boundary-small-beta} 
 \Exp_x \langle D h_w(Tx), \tilde \Delta \rangle = w \| Tx \|^{w-1} 
  \frac{\mu^\pm (x) \cos \theta_0 \cos \alpha}{s \sigma_2}
\left( \pm A_1 + ( a^\pm A_2 + o(1)) x_1^{\beta^\pm -1}  \right)  
,
 \end{align}
where, for $| \theta_0 | < \pi/2$,
$A_1 =   \sigma^2_2 \tan \alpha + \rho  - s \tan \theta_0$,
and
\begin{align*}
A_2 & =
\sigma_2^2 \beta^\pm  
-  \rho \beta^\pm \tan \alpha 
- (1-w)   s \tan \theta_0 \tan \alpha  
- (1-w) \frac{s \rho}{\sigma_2^2} \tan \theta_0  
 \\
& {} \qquad {} +  s \beta^\pm \tan \theta_0 \tan \alpha
- (1-w) \frac{s^2}{\sigma_2^2}  .
\end{align*}
Now take $\theta_0 = \theta_1$ as given by~\eqref{eq:theta-1}, so that
$s \tan \theta_0 =  \sigma_2^2 \tan \alpha + \rho$.
Then $A_1 = 0$, eliminating the leading order term in~\eqref{eq:h-w-boundary-small-beta}.
Moreover, with this choice of $\theta_0$ we get, after some further cancellation and simplification,
that 
\[ A_2 = \frac{\sigma_2^2 \left( \beta^\pm - (1-w) \bc \right)}{\cos^2 \alpha} ,\]
with $\bc$ as given by~\eqref{eq:betac}. 
Thus with~\eqref{eq:h-w-boundary-small-beta} and~\eqref{eq:f-boundary} we verify~\eqref{eq:f-increment-boundary-small-beta}.

Finally suppose that $\beta^\pm > 1$, and restrict to the case $w \in (0,1/2)$.
Let $\theta_2  \in (-\pi/2,\pi/2)$ be as given by~\eqref{eq:theta-2}.
Then if $x = (0,x_2)$, we have $\theta (Tx) = \theta_2 - \frac{\pi}{2}$
if $x_2 <0$ and  $\theta (Tx) = \theta_2 + \frac{\pi}{2}$ if $x_2 >0$ (see Figure~\ref{fig:transform}).
It follows from~\eqref{eq:T-transform} that
\[ \theta (Tx ) = \theta_2 \pm \frac{\pi}{2} + O ( x_1^{1-\beta^\pm} ), \text{ for } x \in S_B^\pm, \]
as $\| x \| \to \infty$ (and $x_1 \to \infty$).
Now~\eqref{eq:h-w-D-vector} with~\eqref{eq:tilde-drift-boundary-big-beta1} and~\eqref{eq:tilde-drift-boundary-big-beta2} shows that 
\begin{align}
\label{eq:h-boundary-big-beta}
 \Exp_x \langle D h_w (Tx), \tilde \Delta \rangle 
  =  w \| Tx \|^{w-1}   \frac{\mu^\pm (x)}{s\sigma_2}   \Bigl( 
\sigma_2^2 \cos \alpha \cos \left( (1-w) \theta  (Tx) + \theta_0 \right)  \nonumber\\
  - \rho \sin \alpha \cos \left( (1-w) \theta  (Tx) + \theta_0 \right) 
+ s \sin \alpha \sin \left( (1-w) \theta  (Tx) + \theta_0 \right) 
+ O ( x_1^{1-\beta^\pm} ) \Bigr) .
 \end{align}
Set $\phi := (1-w) \frac{\pi}{2}$.
Choose $\theta_0 = \theta_3 - (1-w)\theta_2$,
where $\theta_3 \in (-\pi,\pi)$ satisfies~\eqref{eq:sincos3}.
Then we have that, for $x \in S_B^\pm$,
\begin{align}
\label{eq:cos3}
  \cos \left( (1-w) \theta  (Tx) + \theta_0 \right) 
	& = \cos \left( \theta_3 \pm \phi \right)  + O ( x_1^{1-\beta^\pm} ) \nonumber\\
	& = \cos \phi \cos \theta_3 \mp \sin \phi \sin \theta_3  + O ( x_1^{1-\beta^\pm} ).\end{align}
	Similarly, for $x \in S_B^\pm$,
	\begin{align}
	\label{eq:sin3}
	\sin  \left( (1-w) \theta  (Tx) + \theta_0 \right) =\cos \phi \sin \theta_3  \pm \sin \phi \cos \theta_3 + O ( x_1^{1-\beta^\pm} ). \end{align}
Using~\eqref{eq:cos3} and~\eqref{eq:sin3} in~\eqref{eq:h-boundary-big-beta}, we obtain
\begin{align*}
 \Exp_x \langle D h_w (Tx), \tilde \Delta \rangle 
  =  w \| Tx \|^{w-1}   \frac{\mu^\pm (x)}{s\sigma_2} \left( A_3 \cos \phi \mp A_4 \sin \phi + o(1) \right) ,\end{align*}
where
\begin{align*}
A_3  & = \left( \sigma_2^2 \cos \alpha - \rho \sin \alpha \right) \cos \theta_3 + s \sin \alpha \sin \theta_3 \\
& = \sigma_2 d \cos^2 \theta_3 + \sigma_2 d \sin^2 \theta_3 = \sigma_2 d ,\end{align*}
by~\eqref{eq:sincos3}, and, similarly,
\[ A_4 =  \left( \sigma_2^2 \cos \alpha - \rho \sin \alpha \right) \sin \theta_3 - s \sin \alpha \cos \theta_3 = 0.\]
Then with~\eqref{eq:f-boundary} we obtain~\eqref{eq:f-increment-boundary-big-beta}.
\end{proof}

In the case where $\beta^+, \beta^- <1$ with $\beta^+ \neq \beta^-$, we will in some circumstances
need to modify the function $f_w^\gamma$ so that
it can be made insensitive to the behaviour near the boundary with the smaller of $\beta^+, \beta^-$. To this end,
define for $w, \gamma, \nu, \lambda \in \R$,
\begin{equation}
\label{eq:F-def}
 F_w^{\gamma,\nu} (x) :=  f_w^\gamma (x) + \lambda x_2 \| T x\|^{2\nu}  .\end{equation}
We state a result for the case $\beta^- < \beta^+$; an analogous
result holds if $\beta^+ < \beta^-$.

\begin{lemma}
\label{lem:F-increments}
Suppose that~\eqref{ass:bounded-moments}, \eqref{ass:zero-drift}, \eqref{ass:reflection}, and~\eqref{ass:covariance}
hold, with $p>2$, 
 $\alpha^+  = -\alpha^- = \alpha$ for $| \alpha | < \pi/2$,
and $0 \leq \beta^- < \beta^+ < 1$.  
Let $w, \gamma \in \R$ be such that
$2-p < \gamma w < p$. Take
$\theta_0 = \theta_1 \in (-\pi/2,\pi/2)$
given by~\eqref{eq:theta-1}. Suppose that
\[ \gamma w + \beta^- -2 < 2 \nu <  \gamma w + \beta^+ -2   . \]
Then as $\| x \| \to \infty$ with $x \in S_I$,
\begin{align}
\label{eq:F-increment-interior}
 \Exp [   F^{\gamma,\nu}_w(\xi_{n+1})  - F_w^{\gamma,\nu} ( \xi_n) \mid \xi_n = x  ] 
 = \frac{1}{2}\gamma(\gamma-1) (w^2 + o(1) ) ( h_w(Tx) )^{\gamma-2} \| Tx \|^{2w-2}.
\end{align}
As $\| x \| \to \infty$ with $x \in S^+_B$,
\begin{align}
\label{eq:F-increment-boundary-beta-plus}
& {} \Exp [  F^{\gamma,\nu}_w(\xi_{n+1})  - F^{\gamma,\nu}_w( \xi_n) \mid \xi_n = x ] \nonumber\\
& {} \quad {} = \gamma  w \| Tx \|^{w-1} \left( h_w (Tx ) \right)^{\gamma-1}
  \frac{a^+ \mu^+ (x) \sigma_2 \cos \theta_1 }{s \cos \alpha} \left(\beta^+ - (1-w) \bc  +o(1) \right)
x_1^{\beta^+-1}
 .\end{align}
As $\| x \| \to \infty$ with $x \in S^-_B$,
\begin{align}
\label{eq:F-increment-boundary-beta-minus}
  \Exp [  F^{\gamma,\nu}_w(\xi_{n+1})  - F^{\gamma,\nu}_w( \xi_n) \mid \xi_n = x ]  
  = 
 \lambda \| T x \|^{2\nu} \left( \mu^- (x) \cos \alpha + o(1) \right)
 .\end{align}
\end{lemma}
\begin{proof}
Suppose that $0 \leq \beta^- < \beta^+ < 1$. As in the proof of Lemma~\ref{lem:f-increments}, let $E_x = \{ \| \Delta \| < \| x \|^\delta\}$,
where $\delta \in (0,1)$ satisfies~\eqref{eq:delta-choice}.
Set $v_\nu (x) := x_2 \| T x \|^{2\nu}$.
Then, using Taylor's formula in one variable, for $x, y \in \R^2$ with $y \in B_{r/2} (x)$,
\begin{align*}
 \| x + y \|^{2\nu} & = \| x \|^{2\nu} \left( 1 + \frac{2 \langle x, y \rangle + \| y\|^2}{\| x \|^2} \right)^\nu  = \| x \|^{2\nu} + 2 \nu \langle x,y\rangle \| x\|^{2\nu -2} + R (x,y) ,\end{align*}
where  $| R(x,y) | \leq C \| y \|^2 \| x \|^{2\nu -2}$.
Thus, for $x \in S$ with $y \in B_{r/2}(x)$ and $x + y \in S$, 
\begin{align}
\label{eq:v-taylor}
 v_\nu (x+y) - v_\nu (x) & = (x_2 + y_2) \| T x+ T y \|^{2\nu} - x_2 \| T x \|^{2\nu} \nonumber\\
& = y_2 \| T x \|^{2\nu} + 2 \nu x_2 \langle T x, T y\rangle \| T x\|^{2\nu -2} +  2 \nu y_2 \langle T x, T y\rangle \| T x\|^{2\nu -2} \nonumber\\
& {} \qquad {} + R (x,y ),\end{align}
where now $| R(x,y) | \leq C \| y \|^2  \|  x \|^{2 \nu + \beta^+ -2}$,  
using the fact that both $|x_2|$ and $|y_2|$ are $O ( \| x \|^{\beta^+} )$.
 Taking $x = \xi_0$ and $y = \Delta$ so $Ty = \tilde \Delta$, we obtain
\begin{align}
\label{eq:v-increment}
\Exp_x \bigl[ ( v_\nu ( \xi_1 ) - v_\nu (\xi_0 ) ) \2 { E_x } \bigr]
& = \| T x \|^{2\nu} \Exp_x \bigl[ \Delta_2  \2 { E_x } \bigr]
+ 2 \nu x_2 \| T x\|^{2\nu -2} \Exp_x \bigl[ \langle T x, \tilde \Delta\rangle  \2 { E_x } \bigr] \nonumber\\
& {} \qquad {} + 
2 \nu  \| T x\|^{2\nu -2}  \Exp \bigl[ \Delta_2 \langle T x, \tilde \Delta \rangle \2 { E_x } \bigr] 
 + \Exp \bigl[ R(x, \Delta ) \2 { E_x } \bigr].\end{align}
Suppose that $x \in S_I$. Similarly to~\eqref{eq:drift-small-jump}, we have~$\Exp_x [ \langle T x, \tilde \Delta\rangle  \2 { E_x } ] = o(1)$,
and, by similar arguments using~\eqref{eq:big-jumps},  $\Exp [  \Delta_2  \2 { E_x } ] = o ( \| x \|^{-1})$,
$\Exp_x | \Delta_2 \langle T x, \tilde \Delta \rangle \2 { E^\rc_x } | = o ( \| x \| )$,
and $\Exp_x | R(x, \Delta ) \2 { E_x } | = o ( \|x \|^{2\nu -1} )$, since $\beta^+ <1$.
Also, by~\eqref{eq:T-transform},
\begin{align*}
\Exp_x (  \Delta_2 \langle T x, \tilde \Delta \rangle )
 & = \sigma_2 \Exp_x (  \tilde \Delta_2 \langle T x, \tilde \Delta \rangle ) \\
& = \sigma_2 (T x)_1 \Exp_x ( \tilde \Delta_1 \tilde \Delta_2 ) + \sigma_2 (Tx)_2 \Exp_x ( \tilde \Delta_2^2  ) . \end{align*}
Here, by~\eqref{eq:tilde-interior},    $\Exp_x ( \tilde \Delta_1 \tilde \Delta_2 ) = o (1)$
and $\Exp_x ( \tilde \Delta_2^2  ) = O(1)$, while $ \sigma_2 (Tx)_2 = x_2 = O ( \| x \|^{\beta^+})$.
Thus $\Exp_x (  \Delta_2 \langle T x, \tilde \Delta \rangle ) = o ( \| x \| )$. Hence also
\[ \Exp_x \bigl[ \Delta_2 \langle T x, \tilde \Delta \rangle \2 { E_x } \bigr] = o ( \| x\|) .\]
Thus from~\eqref{eq:v-increment} we get that, for $x \in S_I$,
\begin{equation}
\label{eq:v-increment-interior-small-jump}
 \Exp_x \bigl[ ( v_\nu ( \xi_1 ) - v_\nu (\xi_0 ) ) \2 { E_x } \bigr] =  o ( \| x \|^{2\nu -1} ) .\end{equation}
On the other hand, we use the fact that $| v_\nu (x+y ) - v_\nu (x) | \leq C ( \| x \| + \| y \| )^{2\nu + \beta^+}$ to get
\[  \Exp_x \bigl[ | v_\nu ( \xi_1 ) - v_\nu (\xi_0 ) | \2 { E^\rc_x } \bigr]
\leq C \Exp_x \bigl[ \| \Delta \|^{(2 \nu+\beta^+) /\delta} \2 { E^\rc_x } \bigr] .\]
Here $2\nu +\beta^+ < 2 \nu + 1 < \gamma w < \delta p$, by choice of $\nu$ and~\eqref{eq:delta-choice}, so we may apply~\eqref{eq:big-jumps}
with $q = (2 \nu+\beta^+) /\delta$ to get
\begin{equation}
\label{eq:v-increment-interior-big-jump}  
\Exp_x \bigl[ | v_\nu ( \xi_1 ) - v_\nu (\xi_0 ) | \2 { E^\rc_x } \bigr] = O ( \| x \|^{2\nu+\beta^+-\delta p} ) = o ( \| x \|^{2\nu-1} ) ,\end{equation}
since $\delta p > 2$, by~\eqref{eq:delta-choice}. Combining~\eqref{eq:v-increment-interior-small-jump}, \eqref{eq:v-increment-interior-big-jump} and~\eqref{eq:f-increment-interior},
we obtain~\eqref{eq:F-increment-interior}, provided that $2\nu -1 < \gamma w -2$, 
which is the case since  $2\nu < \gamma w + \beta^+ - 2$ and $\beta^+ < 1$.

Now suppose that $x \in S_B^\pm$. We 
truncate~\eqref{eq:v-taylor} to see that, for $x \in S$ with $y \in B_{r/2}(x)$ and $x + y \in S$, 
\[ 
 v_\nu (x+y) - v_\nu (x)   = y_2 \| T x \|^{2\nu}  + R (x,y ), \]
where now $| R(x,y) | \leq C \| y \| \| x \|^{2\nu + \beta^\pm -1}$, 
using the fact that for $x \in S_B^\pm$, $|x_2| = O ( \| x \|^{\beta^\pm} )$.
It follows that, for $x \in S_B^\pm$,
\[
\Exp_x \bigl[ ( v_\nu ( \xi_1 ) - v_\nu (\xi_0 ) ) \2 {E_x } \bigr]
=  \| T x \|^{2\nu} \Exp_x \bigl[ \Delta_2 \2 {E_x } \bigr]
+ O ( \| x \|^{2\nu + \beta^\pm -1} ).
\]
By~\eqref{eq:big-jumps} and~\eqref{eq:delta-choice}
we have that $\Exp [ |\Delta_2| \2 {E^\rc_x } ] = O( \| x \|^{-\delta(p-1)} ) = o( \| x\|^{-1} )$, 
while
if $x \in S_B^\pm$, then, by~\eqref{eq:drift-boundary-small-beta2},
 $\Exp_x \Delta_2 = \mp \mu^\pm (x) \cos \alpha + O( \|x\|^{\beta^\pm-1})$. 
On the other hand, the estimate~\eqref{eq:v-increment-interior-big-jump}  still applies,
so we get, for $x \in S_B^\pm$,
\begin{equation}
\label{eq:v-boundary}
\Exp_x [ v_\nu ( \xi_1 ) - v_\nu (\xi_0 ) ]
  = \mp \| T x \|^{2\nu} \mu^\pm (x) \cos \alpha + O ( \| x \|^{2\nu +\beta^\pm-1 } ) .\end{equation}
If we choose $\nu$ such that $2 \nu < \gamma w + \beta^+ - 2$,
then we  combine~\eqref{eq:v-boundary} and~\eqref{eq:f-increment-boundary-small-beta}
to get~\eqref{eq:F-increment-boundary-beta-plus}, since the 
term from~\eqref{eq:f-increment-boundary-small-beta} dominates.
 If we choose $\nu$ such that $2 \nu > \gamma w + \beta^- - 2$,
then the term from~\eqref{eq:v-boundary} dominates that from~\eqref{eq:f-increment-boundary-small-beta},
and we get~\eqref{eq:F-increment-boundary-beta-minus}.
\end{proof}

In the critically recurrent cases, where $\max ( \beta^+,\beta^- ) = \bc \in (0,1)$
or $\beta^+,\beta^- >1$, in which no passage-time moments exist, the functions of polynomial
growth based on $h_w$ as defined at~\eqref{eq:h-w-def} are not sufficient to prove recurrence.
Instead we need functions which grow more slowly. For $\eta \in \R$ let
\begin{equation}
\label{eq:h-def}
h (x) := h(r,\theta ) := \log r + \eta \theta , \text{ and } \ell (x) := \log h( Tx) ,\end{equation}
where we understand $\log y$ to mean $\max (1,\log y)$. 
The function $h$ is again harmonic (see below) and was used in the context of reflecting Brownian motion in a wedge in~\cite{vw}.
Set
\begin{equation}
\label{eq:eta-0}
\eta_0 := \eta_0 ( \Sigma,\alpha ) := \frac{\sigma_2^2 \tan \alpha + \rho}{s},
\text{ and }
\eta_1 := \eta_1 (\Sigma, \alpha):= \frac{\sigma_1^2 \tan \alpha - \rho}{s} .\end{equation}

\begin{lemma}
\label{lem:ell-increments}
Suppose that~\eqref{ass:bounded-moments}, \eqref{ass:zero-drift-plus}, \eqref{ass:reflection}, and~\eqref{ass:covariance-plus}
hold, with $p>2$, $\eps >0$,
 $\alpha^+  = -\alpha^- = \alpha$ for $| \alpha | < \pi/2$,
and $\beta^+, \beta^- \geq 0$.  
For any $\eta \in \R$, as $\| x \| \to \infty$ with $x \in S_I$,
\begin{align}
\label{eq:ell-increment-interior}
\Exp [ \ell ( \xi_{n+1}) - \ell ( \xi_n) \mid \xi_n = x ]  =  - \frac{1+\eta^2 +o(1)}{2 \| Tx\|^2(\log \| Tx\|)^2}   .\end{align}
If $0 \leq \beta^\pm < 1$, take $\eta = \eta_0$ as defined at~\eqref{eq:eta-0}. Then,
as $\| x \| \to \infty$ with $x \in S^\pm_B$,
\begin{align}
\label{eq:ell-increment-boundary-small-beta}
& {} \Exp [ \ell (\xi_{n+1}) - \ell (\xi_n) \mid \xi_n = x ] \nonumber\\
& {} \qquad {}
=  \frac{\sigma_2^2 a^\pm \mu^\pm (x)}{s^2 \cos \alpha} \frac{1}{\|T x\|^2 \log\|T x\|} \left( ( \beta^\pm - \bc ) x^{\beta^\pm}_1 + O ( \| x \|^{2\beta^\pm-1} ) + O(1) \right)  .\end{align}
If $\beta^\pm > 1$, take $\eta = \eta_1$ as defined at~\eqref{eq:eta-0}. Then 
as $\| x \| \to \infty$ with $x \in S^\pm_B$,
\begin{align}
\label{eq:ell-increment-boundary-big-beta}
& {} \Exp [ \ell (\xi_{n+1}) - \ell (\xi_n) \mid \xi_n = x ] \nonumber\\
& {} \qquad {} =   
\frac{\mu^\pm (x)}{ s^2 \cos \alpha} 
\frac{x_1}{\|T x\|^2 \log\|T x\|}
 \left(  \sigma_1^2  \sin^2 \alpha + \sigma_2^2  \cos^2 \alpha    - \frac{\sigma_1^2}{\beta^\pm}
-  \rho \sin 2 \alpha + o(1) \right)   .\end{align}
\end{lemma}
\begin{proof}
Given $\eta \in \R$, for $r_0 = r_0 (\eta ) = \exp ( \re + | \eta | \pi )$,
we have from~\eqref{eq:eta-0} that both $h$ and $\log h$
are infinitely differentiable in the domain $\cR_{r_0} := \{ x \in \R^2 : x_1 > 0, \, r (x) > r_0 \}$.
Differentiating~\eqref{eq:eta-0} and using~\eqref{eq:polar-diff} we obtain, for $x \in \cR_{r_0}$,
\begin{equation}
\label{eq:h-derivatives}
 D_1 h(x) = \frac{1}{r} \left( \cos \theta - \eta \sin \theta \right) , \text{ and } D_2 h(x) = \frac{1}{r}  \left( \sin \theta + \eta \cos \theta \right) .\end{equation}
We verify that $h$ is harmonic in $\cR_{r_0}$, since
\[ D_1^2 h(x) =  \frac{\eta \sin 2 \theta}{r^2}  - \frac{\cos 2 \theta}{r^2} = - D_2^2 h(x) .\]
Also, for any $i,j,k$, $| D_i D_j D_k h(x) | = O (r^{-3} )$. Moreover, $D_i \log h(x) = (h(x))^{-1} D_i h(x)$,  
\[ D_i D_j \log h(x) = \frac{D_i D_j h(x)}{h(x)} - \frac{ (D_i h(x)) (D_j h(x))}{(h(x))^2} ,\]
and $| D_i D_j D_k \log h(x) | = O (r^{-3} (\log r )^{-1})$.
Recall that $Dh(x)$ is the vector
function whose components are $D_i h(x)$.
Then Taylor's formula~\eqref{eq:taylor-second-order} together with the harmonic property of $h$ shows that
for $x \in \cR_{2r_0}$ and $y \in B_{r/2} (x)$,
\begin{align}
\label{eq:h-taylor}
 \log h (x +y ) 
& = \log h(x) + \frac{\langle D h(x), y \rangle}{h(x)}
+ \frac{(y_1^2 - y_2^2) D_1^2 h(x)}{2h(x)} + \frac{ y_1 y_2 D_{1} D_{2} h(x)}{h(x)} \nonumber\\
& {} \qquad {} - \frac{\langle D h(x), y \rangle^2}{2(h(x))^2}
+ R (x,y)
,\end{align}
where $| R(x,y) | \leq C \| y \|^3 \| x \|^{-3} (\log \| x \|)^{-1}$ for some constant $C< \infty$,
all $y \in B_{r/2} (x)$, and all $\| x \|$ sufficiently large.
As in the proof of Lemma~\ref{lem:f-increments}, let $E_x = \{ \| \Delta \| < \| x \|^\delta\}$
for $\delta \in (\frac{2}{p},1)$.
Then applying the expansion in~\eqref{eq:h-taylor} to $\log h(Tx + \tilde \Delta)$, conditioning on $\xi_0 = x$, and taking expectations,
we obtain, for $\| x \|$ sufficiently large,
\begin{align}
\label{eq:ell-increment-small-jump}
& {} \Exp_x \bigl[ ( \ell (\xi_1) - \ell (\xi_0) ) \2 { E_x} \bigr]  
 = \frac{\Exp_x \bigl[ \langle D h(Tx), \tilde \Delta \rangle \2 { E_x} \bigr]  }{h(Tx)}
+ \frac{D_1^2 h(Tx) \Exp_x \bigl[ (\tilde \Delta_1^2 - \tilde \Delta_2^2)\2 { E_x} \bigr] }{2h(Tx)} \nonumber\\
& {} \qquad {} + \frac{  D_{1} D_{2} h(Tx) \Exp_x \bigl[ \tilde \Delta_1 \tilde \Delta_2 \2 { E_x} \bigr]}{h(Tx)}  - \frac{\Exp_x \bigl[ \langle D h(Tx), \tilde \Delta \rangle^2 \2 { E_x} \bigr]}{2(h(Tx))^2}
+ \Exp_x \bigl[ R (Tx,\tilde \Delta)\2 { E_x} \bigr]
.\end{align}
Let $p' \in (2,3]$ be such that~\eqref{eq:p-moments} holds. Then
\[ \Exp_x \bigl| R (Tx,\tilde \Delta)\2 { E_x} \bigr| \leq C \| x \|^{-3+(3-p')\delta} \Exp_x ( \| \Delta \|^{p'} ) = O ( \| x \|^{-2-\eps'} ) ,\]
for some $\eps' >0$. 

Suppose that $x \in S_I$. By~\eqref{eq:tilde-interior-plus}, 
$\Exp_x ( \tilde \Delta_1 \tilde \Delta_2 ) = O ( \| x \|^{-\eps} )$ and, by~\eqref{eq:big-jumps},
$\Exp_x | \tilde \Delta_1 \tilde \Delta_2\2 {E_x^\rc} | \leq C \Exp [ \| \Delta \|^2 \2 {E_x^\rc} ]
= O ( \| x \|^{-\eps'} )$,
for some $\eps' >0$. Thus $\Exp_x ( \tilde \Delta_1 \tilde \Delta_2\2 {E_x} ) = O ( \| x \|^{-\eps'} )$. A similar argument gives the same bound for $\Exp_x  [ ( \tilde \Delta_1^2 - \tilde \Delta_2^2) \2 { E_x}  ]$.
Also, from~\eqref{eq:tilde-interior-plus} and~\eqref{eq:h-derivatives}, $\Exp_x ( \langle D h(Tx), \tilde \Delta \rangle ) =  O ( \| x \|^{-2-\eps} )$
and, by~\eqref{eq:big-jumps}, $\Exp_x | \langle D h(Tx), \tilde \Delta \rangle \2 {E_x^\rc} | = O ( \| x \|^{-2-\eps'} )$
for some $\eps' >0$. Hence $\Exp_x  [ \langle D h(Tx), \tilde \Delta \rangle \2 { E_x} ] = O ( \| x \|^{-2-\eps'} )$.
Finally, by~\eqref{eq:tilde-interior-plus} and~\eqref{eq:h-derivatives}, 
\begin{align*} \Exp_x \langle D h(Tx), \tilde \Delta \rangle^2 & = \Exp_x \left( ( Dh (Tx) )^\tra \tilde \Delta \tilde \Delta^\tra Dh (Tx) \right)\\
& = ( Dh (Tx) )^\tra  Dh (Tx) + O ( \| x \|^{-2-\eps} ) \\
& = ( D_1 h(Tx) )^2 + ( D_2 h(Tx) )^2  + O ( \| x \|^{-2-\eps} ),\end{align*}
while, by~\eqref{eq:big-jumps}, $ \Exp_x | \langle D h(Tx), \tilde \Delta \rangle^2 \2 {E_x^\rc } | = O ( \| x \|^{-2-\eps'} )$.
Putting all these estimates into~\eqref{eq:ell-increment-small-jump} gives
\begin{align*}
  \Exp_x \bigl[ ( \ell (\xi_1) - \ell (\xi_0) ) \2 { E_x} \bigr]  
 =   - \frac{ ( D_1 h(Tx) )^2 + ( D_2 h(Tx) )^2}{2(h(Tx))^2} + O ( \| x \|^{-2-\eps'} )
,\end{align*}
for some $\eps' >0$.
On the other hand, for all $\| x \|$ sufficiently large,
$| \ell (x+y) - \ell (x) | \leq C \log \log \| x \| + C \log \log \| y \|$.
For any $p>2$ and $\delta \in (\frac{2}{p}, 1)$, we may (and do) choose
 $q > 0$ sufficiently small such that $\delta (p-q) >2$, and then, by~\eqref{eq:big-jumps},
\begin{align}
\label{eq:ell-increment-big-jump}
  \Exp_x \bigl[ ( \ell (\xi_1) - \ell (\xi_0) ) \2 { E^\rc_x} \bigr]  
& \leq C \Exp_x \bigl[ \| \Delta \|^{q} \2 { E^\rc_x} \bigr] \nonumber\\
& = O ( \| x \|^{-\delta(p-q)} ) = O ( \| x \|^{-2-\eps'} ) ,\end{align}
for some $\eps'>0$.
Thus we conclude that
\begin{align*}
  \Exp_x \bigl[   \ell (\xi_1) - \ell (\xi_0)  \bigr]  
 =   - \frac{ ( D_1 h(Tx) )^2 + ( D_2 h(Tx) )^2}{2(h(Tx))^2} + O ( \| x \|^{-2-\eps'} ),
\end{align*}
for some $\eps' >0$.
Then~\eqref{eq:ell-increment-interior} follows from~\eqref{eq:h-derivatives}.

Next suppose that $x \in S_B$. Truncating~\eqref{eq:h-taylor}, we have for  $x \in \cR_{2r_0}$ and $y \in B_{r/2} (x)$,
\[ \log h (x+y ) = \log h (x) + \frac{ \langle D h(x), y \rangle }{h(x)} + R(x,y) ,\]
where now $|R(x,y)| \leq C \| y \|^2 \| x \|^{-2} (\log \| x \|)^{-1}$ for $\| x \|$ sufficiently large.
It follows that
\begin{align*}
\Exp_x \bigl[ ( \ell (\xi_1) - \ell (\xi_0) ) \2 { E_x } \bigr]  
& = \frac{\Exp_x  \bigl[  \langle D h(Tx), \tilde \Delta \rangle \2 { E_x } \bigr] + O ( \| x \|^{-2})}{h(Tx)} 
.\end{align*}
Using~\eqref{eq:ell-increment-big-jump} and the fact that
$\Exp_x | \langle D h(Tx), \tilde \Delta \rangle \2 {E_x^\rc} | = O ( \| x \|^{-2-\eps'} )$ (as above),
we obtain  
\begin{align}
\label{eq:ell-boundary}
\Exp_x \bigl[   \ell (\xi_1) - \ell (\xi_0)   \bigr]  
& = \frac{\Exp_x  \bigl[  \langle D h(Tx), \tilde \Delta \rangle \bigr] + O ( \| x \|^{-2})}{h(Tx)} 
.\end{align}
From~\eqref{eq:h-derivatives} we have
\[ D h (x) = \frac{1}{\| x \|^2} \begin{pmatrix} x_1 - \eta x_2 \\
x_2 + \eta x_1 \end{pmatrix}, \text{ and hence }
D h (Tx) = \frac{1}{\| Tx \|^2} \begin{pmatrix} \frac{\sigma_2}{s} x_1 -\frac{\rho}{s\sigma_2} x_2 - \frac{\eta}{\sigma_2} x_2 \\
\frac{1}{\sigma_2}x_2 + \frac{\eta\sigma_2}{s} x_1 - \frac{\eta \rho}{s\sigma_2} x_2 \end{pmatrix} ,\]
using~\eqref{eq:T-transform}. If $\beta^\pm <1$ and $x \in S_B^\pm$, 
we have from~\eqref{eq:tilde-drift-boundary-small-beta1} and~\eqref{eq:tilde-drift-boundary-small-beta2} that
\begin{align*}
& \Exp_x \langle D h(Tx), \tilde \Delta \rangle \\
& {} \quad {} = \frac{\mu^\pm (x)}{s^2} \frac{1}{\| Tx \|^2} \biggl\{ 
 a^\pm \Bigl[  \left( s \eta (\beta^\pm-1) - \rho (1+\beta^\pm) \right) \sin \alpha + 
\left( \sigma_2^2 \beta^\pm  - \sigma_1^2 \right) \cos \alpha  \Bigr]
x_1^{\beta^\pm}  \\
& {} \qquad {}   \pm \Bigl[ \sigma_2^2 \sin \alpha + (\rho - s \eta ) \cos \alpha \Bigr] x_1
+ O ( x_1^{2\beta^\pm -1} ) + O(1) \biggr\} .
 \end{align*}
Taking $\eta = \eta_0$ as given by~\eqref{eq:eta-0},
the $\pm x_1$ term vanishes, and, after simplification, we get
\begin{align}
\label{eq:ell-boundary-small-beta} 
 \Exp_x \langle D h(Tx), \tilde \Delta \rangle 
= \frac{\sigma_2^2 a^\pm \mu^\pm (x)}{s^2 \cos \alpha} \frac{1}{\| Tx \|^2} \left(
\left(  \beta^\pm - \bc \right)
x_1^{\beta^\pm}  
+ O ( x_1^{2\beta^\pm -1} ) + O(1) \right) .
 \end{align}
Using~\eqref{eq:ell-boundary-small-beta} in~\eqref{eq:ell-boundary} gives~\eqref{eq:ell-increment-boundary-small-beta}.

On the other hand, if $\beta^\pm >1$  and $x \in S_B^\pm$, we 
have from~\eqref{eq:tilde-drift-boundary-big-beta1} and~\eqref{eq:tilde-drift-boundary-big-beta2} that
\begin{align*}
& \Exp_x \langle D h(Tx), \tilde \Delta \rangle \\
& {} \quad {} = \frac{\mu^\pm (x)}{s^2} \frac{1}{\| Tx \|^2} \biggl\{ 
 \frac{1}{\beta^\pm} \Bigl[  \left( s \eta (\beta^\pm-1) - \rho (1+\beta^\pm) \right) \sin \alpha + 
 \left( \sigma_2^2 \beta^\pm  - \sigma_1^2 \right) \cos \alpha  \Bigr]
x_1 \\
& {} \qquad {}   \pm a^\pm \Bigl[ \sigma_1^2 \sin \alpha - (\rho + s \eta ) \cos \alpha \Bigr] x_1^{\beta^\pm}
+ O ( x_1^{2-\beta^\pm} ) + O(1) \biggr\} .
 \end{align*}
Taking $\eta = \eta_1$ as given by~\eqref{eq:eta-0},
the $\pm x_1^{\beta^\pm}$ term vanishes, and we get
\begin{align}
\label{eq:ell-boundary-big-beta}
\Exp_x \langle D h(Tx), \tilde \Delta \rangle 
=
 \frac{\mu^\pm (x)}{s^2 \cos \alpha}
 \frac{x_1}{\| Tx \|^2}\! \left(\!
  \sigma_1^2 \sin^2\alpha + \sigma_2^2 \cos^2 \alpha - \frac{\sigma_1^2}{\beta^\pm}  - \rho \sin 2\alpha 
		+ o(1) \! \right) ,
 \end{align}
as $\| x \| \to \infty$ (and $x_1 \to \infty$).
Then using~\eqref{eq:ell-boundary-big-beta} in~\eqref{eq:ell-boundary} gives~\eqref{eq:ell-increment-boundary-big-beta}.
\end{proof}

The function $\ell$ is not by itself enough to prove recurrence in the critical cases,
because the estimates in Lemma~\ref{lem:ell-increments} do not guarantee
that $\ell$ satisfies a supermartingale condition for all parameter values of interest.
To proceed, we modify the function slightly to improve its properties near the boundary.
In the case where $\max (\beta^+,\beta^-) = \bc \in (0,1)$,
the following function will be used to prove recurrence,
\[ g_\gamma (x) := g_\gamma (r, \theta) : = \ell (x) + \frac{\theta^2}{(1+r)^{\gamma}} , \]
where the parameter $\eta$ in $\ell$ is chosen as $\eta = \eta_0$ as given by~\eqref{eq:eta-0}.

\begin{lemma}
\label{lem:g-gamma-increments}
Suppose that~\eqref{ass:bounded-moments}, \eqref{ass:zero-drift-plus}, \eqref{ass:reflection}, and~\eqref{ass:covariance-plus}
hold, with $p>2$, $\eps >0$,
 $\alpha^+  = -\alpha^- = \alpha$ for $| \alpha | < \pi/2$,
and $\beta^+, \beta^- \in (0,1)$ with $\beta^+, \beta^- \leq \bc$.
Let $\eta = \eta_0$, and suppose 
\[ 0 < \gamma < \min  ( \beta^+, \beta^-, 1 - \beta^+, 1-\beta^- , p-2 ). \] 
 Then as $\| x \| \to \infty$ with $x \in S_I$,
\begin{align}
\label{eq:g-gamma-increment-interior}
\Exp [ g_\gamma (\xi_{n+1}) - g_\gamma (\xi_n) \mid \xi_n = x] =
  - \frac{1+\eta^2 + o(1)}{2 \| Tx\|^2(\log \| Tx\|)^2}    .\end{align}
Moreover, 
as $\| x \| \to \infty$ with $x \in S_B^\pm$,  
\begin{align}
\label{eq:g-gamma-increment-boundary}
\Exp [ g_\gamma (\xi_{n+1}) - g_\gamma (\xi_n) \mid \xi_n = x]
\leq   - 2 a^\pm \mu^\pm (x) (\cos \alpha +o(1)) \| x \|^{\beta^\pm -2 - \gamma}  
.\end{align}
\end{lemma}
\begin{proof}
Set $u_\gamma ( x) := u_\gamma (r,\theta) := \theta^2 (1+r)^{-\gamma}$, and note that, by~\eqref{eq:polar-diff}, for $x_1 >0$,
\begin{align*}
 D_1 u_\gamma (x)   = - \frac{2\theta \sin \theta}{r (1+r)^{\gamma}} - \frac{\gamma \theta^2 \cos \theta}{(1+r)^{1+\gamma}} , ~~~  D_2 u_\gamma (x)  =   \frac{2\theta \cos \theta}{r (1+r)^{\gamma}} - \frac{\gamma \theta^2 \sin \theta}{(1+r)^{1+\gamma}} ,
\end{align*}
and $| D_i D_j u_\gamma (x) | = O (r^{-2-\gamma})$ for any $i,j$. 
So, by Taylor's formula~\eqref{eq:taylor-second-order}, for all $y \in B_{r/2} (x)$,
\[ u_\gamma (x +y ) = u_\gamma (x) + \langle D u_\gamma (x), y \rangle + R(x,y),\]
where $| R(x,y) | \leq C \| y \|^2 \| x \|^{-2-\gamma}$ for all $\| x \|$ sufficiently large.
Once more define the event $E_x = \{ \| \Delta \| < \| x \|^{\delta} \}$, where now $\delta \in ( \frac{2+\gamma}{p},1)$.
Then
\[ \Exp_x \bigl[ ( u_\gamma (\xi_1) - u_\gamma (\xi_0) ) \2 {E_x } \bigr] = \Exp_x \bigl[ \langle D u_\gamma (x), \Delta \rangle \2 {E_x } \bigr] +  O (\|x\|^{-2-\gamma}) .\]
Moreover, 
$\Exp_x | \langle D u_\gamma (x), \Delta \rangle \2 {E^\rc_x } | \leq C \| x \|^{-1-\gamma} \Exp_x ( \| \Delta \| \2 {E^\rc_x } )
= O ( \| x \|^{-2-\gamma})$, by~\eqref{eq:big-jumps} and the fact that $\delta > \frac{2}{p} > \frac{1}{p-1}$. Also, since $u_\gamma$ is uniformly bounded,
\[  \Exp_x \bigl[ | u_\gamma (\xi_1) - u_\gamma (\xi_0) | \2 {E^\rc_x } \bigr] \leq C \Pr_x ( E^\rc_x ) = O ( \| x \|^{-p\delta} ) ,\]
by~\eqref{eq:big-jumps}. Since $p\delta > 2+\gamma$, it follows that
\begin{equation}
\label{eq:u-increment}
 \Exp_x \bigl[  u_\gamma (\xi_1) - u_\gamma (\xi_0)   \bigr] = \Exp_x  \langle D u_\gamma (x), \Delta \rangle   +  O (\|x\|^{-2-\gamma}) .\end{equation}
For $x \in S_I$, it follows from~\eqref{eq:u-increment} and~\eqref{ass:zero-drift-plus} that $\Exp_x [ u_\gamma (\xi_1) - u_\gamma (\xi_0) ] = O (\|x\|^{-2-\gamma})$, and 
combining this with~\eqref{eq:ell-increment-interior} we get~\eqref{eq:g-gamma-increment-interior}.

Let $\beta = \max (\beta^+, \beta^- ) < 1$. For $x \in S$, 
$ | \theta (x) | = O ( r^{\beta -1})$ as $\| x \| \to \infty$, so~\eqref{eq:u-increment} gives
\[ \Exp_x [ u_\gamma (\xi_1) - u_\gamma (\xi_0) ] = \frac{2\theta \cos \theta \Exp_x \Delta_2}{\| x \| (1+\|x\|)^{\gamma}} + 
O (\|x\|^{2\beta -3-\gamma} ) + O (\|x\|^{-2-\gamma}) .\]
If $x \in S_B^\pm$ then
$\theta = \pm a^\pm (1+o(1)) x_1^{\beta^\pm-1}$ and, by~\eqref{eq:drift-boundary-small-beta2}, 
 $\Exp_x \Delta_2 = \mp \mu^\pm (x) \cos \alpha +o(1)$, so 
\begin{equation}
\label{eq:u-increment-boundary}
 \Exp_x [ u_\gamma (\xi_1) - u_\gamma (\xi_0) ] = - 2 a^\pm \mu^\pm (x) (\cos \alpha +o(1)) \| x \|^{\beta^\pm -2 - \gamma}   . \end{equation}
For $\eta = \eta_0$ and $\beta^+, \beta^- \leq \bc$, we
have from~\eqref{eq:ell-increment-boundary-small-beta} that
\[ 
\Exp_x [ \ell (\xi_{1}) - \ell (\xi_0)  ] 
\leq  \frac{1}{\|T x\|^2 \log\|T x\|}  \left( O ( \| x \|^{2\beta^\pm-1} ) + O(1) \right)  .\]
Combining this with~\eqref{eq:u-increment-boundary}, we obtain~\eqref{eq:g-gamma-increment-boundary},
provided that we choose $\gamma$ such that $\beta^\pm -2 - \gamma > 2 \beta^\pm -3$ and $\beta^\pm -2 - \gamma > -2$, that is,
$\gamma < 1 - \beta^\pm$ and $\gamma < \beta^\pm$.
\end{proof}

In the case where $\beta^+,\beta^->1$, 
we will use the function
\[ w_\gamma (x) := \ell (x) - \frac{x_1}{(1+\| x \|^2)^{\gamma}} , \]
where the parameter $\eta$ in $\ell$ is now chosen as $\eta = \eta_1$ as defined at~\eqref{eq:eta-0}.
A similar function was used in~\cite{afm}.

\begin{lemma}
\label{lem:w-increments}
Suppose that~\eqref{ass:bounded-moments}, \eqref{ass:zero-drift-plus}, \eqref{ass:reflection}, and~\eqref{ass:covariance-plus}
hold, with $p>2$, $\eps >0$,
 $\alpha^+  = -\alpha^- = \alpha$ for $| \alpha | < \pi/2$,
and $\beta^+, \beta^- >1$
Let $\eta = \eta_1$, and suppose that 
\[ \frac{1}{2} < \gamma < \min \left( 1 - \frac{1}{2\beta^+}, 1 - \frac{1}{2\beta^-} , \frac{p-1}{2} \right) .\]
 Then as $\| x \| \to \infty$ with $x \in S_I$,
\begin{align}
\label{eq:w-increment-interior}
\Exp [ w_\gamma ( \xi_{n+1} ) - w_\gamma ( \xi_n ) \mid \xi_n = x ] =
 - \frac{1+\eta^2 + o(1)}{2 \| Tx\|^2(\log \| Tx\|)^2}  .\end{align}
Moreover,   
as $\| x \| \to \infty$ with $x \in S^\pm_B$,
\begin{align} 
\label{eq:w-increment-boundary} 
\Exp [ w_\gamma ( \xi_{n+1} ) - w_\gamma ( \xi_n ) \mid \xi_n = x ] 
= - \frac{\mu^\pm (x) \cos \alpha + o(1)}{\| x \|^{2\gamma}} .\end{align}
\end{lemma}
\begin{proof}
Let $q_\gamma (x) := x_1 (1+\| x \|^2)^{-\gamma}$. Then
\[ D_1 q_\gamma (x) = \frac{1}{(1+\| x\|^2)^{\gamma}} - \frac{2\gamma x_1^2}{(1+\| x \|^2)^{1+\gamma}},
~~~
D_2 q_\gamma (x) = - \frac{2\gamma x_1 x_2}{(1+\| x\|^2)^{1+\gamma}} ,\]
and $| D_i D_j q_\gamma (x) | = O ( \| x \|^{-1-2\gamma} )$
for any $i,j$. Thus by Taylor's formula, for $y \in B_{r/2} (x)$,
\[ q_\gamma (x + y ) - q_\gamma (x) = \langle D q_\gamma (x), y \rangle + R (x,y), \]
where $| R(x,y) | \leq C \| y \|^2 \| x \|^{-1-2\gamma}$ for $\| x \|$ sufficiently large.
Once more let $E_x = \{ \| \Delta \| < \| x \|^{\delta} \}$, where now we take $\delta \in (\frac{1+2\gamma}{p}, 1)$.
Then
\[ \Exp_x \bigl[ ( q_\gamma (\xi_1 ) - q_\gamma (\xi_0) ) \2 { E_x} \bigr] = \Exp_x \bigl[ \langle D q_\gamma (x), \Delta \rangle  \2 { E_x} \bigr] + O ( \| x \|^{-1-2\gamma} ) .\]
Moreover, we get from~\eqref{eq:big-jumps} that
$\Exp_x | \langle D q_\gamma (x), \Delta \rangle  \2 { E^\rc_x} | = O ( \| x \|^{-2\gamma-\delta(p-1) })$,
where $\delta (p-1) > 2\gamma > 1$, and, since $q_\gamma$ is uniformly bounded for $\gamma > 1/2$,
\[ \Exp_x \bigl[ ( q_\gamma (\xi_1 ) - q_\gamma (\xi_0) ) \2 { E^\rc_x} \bigr] = O ( \| x \|^{-p\delta} ) ,\]
where $p\delta > 1 +2\gamma$. Thus
\begin{equation}
\label{eq:q-increment}
 \Exp_x \bigl[   q_\gamma (\xi_1 ) - q_\gamma (\xi_0)   \bigr] = \Exp_x  \langle D q_\gamma (x), \Delta \rangle  + O ( \| x \|^{-1-2\gamma} ) .\end{equation}
If $x \in S_I$, then~\eqref{ass:zero-drift-plus} gives $\Exp_x \langle D q_\gamma (x), \Delta \rangle = O ( \| x\|^{-1-2\gamma})$ and with~\eqref{eq:ell-increment-interior}
we get~\eqref{eq:w-increment-interior}, since $\gamma > 1/2$.
On the other hand, suppose that $x \in S_B^\pm$ and $\beta^\pm >1$.
Then $\| x\| \geq c x_1^{\beta^\pm}$ for some $c>0$, 
so $x_1 = O ( \| x \|^{1/\beta^\pm} )$. So, by~\eqref{eq:q-increment},
\[ \Exp_x [ q_\gamma (\xi_1) - q_\gamma (\xi_0) ] 
= \frac{\Exp_x \Delta_1}{(1+\| x \|^2)^{\gamma}} + O \left( \|x\|^{\frac{1}{\beta^\pm} - 1 -2\gamma} \right).\]
Moreover,  by~\eqref{eq:drift-boundary-big-beta1}, 
 $\Exp_x \Delta_1 = \mu^\pm (x) \cos \alpha + o(1)$. Combined with~\eqref{eq:ell-increment-boundary-big-beta},
this yields~\eqref{eq:w-increment-boundary}, provided that $2 \gamma \leq 2 - (1/\beta^\pm)$,
again using the fact that $x_1 = O ( \| x \|^{1/\beta^\pm} )$. This completes the proof.
\end{proof}

\section{Proofs of main results}
\label{sec:proofs}

We obtain our recurrence classification and quantification of passage-times  
via Foster--Lyapunov criteria (cf.~\cite{foster}). As we do not assume any irreducibility, the most convenient
form of the criteria are those for discrete-time adapted processes presented in~\cite{mpw}. However,
the recurrence criteria in \cite[\S 3.5]{mpw} are formulated for processes
on $\RP$, and, strictly, do not apply directly here. Thus we present appropriate generalizations here, as they may also be useful elsewhere.
The following recurrence result is based on Theorem~3.5.8 of~\cite{mpw}.

\begin{lemma}
\label{lem:recurrence}
Let $X_0, X_1, \ldots$ be a stochastic process on $\R^d$ adapted to a filtration $\cF_0, \cF_1, \ldots$.
Let $f : \R^d \to \RP$ be such that $f(x) \to \infty$ as $\| x \| \to \infty$, and $\Exp f(X_0) < \infty$. Suppose that there
exist $r_0 \in \RP$ and $C < \infty$ for which, for all $n \in \ZP$,
\begin{align*}
\Exp [ f (X_{n+1} ) - f(X_n) \mid \cF_n] & \leq 0, \text{ on } \{ \| X_n \| \geq r_0 \}; \\
\Exp [ f (X_{n+1} ) - f(X_n) \mid \cF_n] & \leq C, \text{ on } \{ \| X_n \| < r_0 \}.\end{align*}
Then if $\Pr ( \limsup_{n \to \infty} \| X_n \| = \infty ) =1$, we have that $\Pr ( \liminf_{n \to \infty} \| X_n \| \leq r_0 ) = 1$.
\end{lemma}
\begin{proof}
By hypothesis, $\Exp f(X_n) < \infty$ for all $n$. Fix $n \in \ZP$ and let $\lambda_{n} := \min \{ m \geq n : \| X_m \| \leq r_0 \}$
and, for some $r > r_0$, set $\sigma_n := \min \{ m \geq n : \| X_m \| \geq  r \}$.
Since $\limsup_{n \to \infty} \| X_n \| = \infty$ a.s., we have that $\sigma_n < \infty$, a.s.
Then $f(X_{m \wedge \lambda_n \wedge \sigma_n } )$, $m \geq n$, is a non-negative supermartingale
with $\lim_{m \to \infty} f(X_{m \wedge \lambda_n \wedge \sigma_n } ) = f(X_{\lambda_n \wedge \sigma_n } )$, a.s. 
By Fatou's lemma
and the fact that $f$ is non-negative,
\[ \Exp f (X_n) \geq \Exp f(X_{\lambda_n \wedge \sigma_n } ) \geq \Pr ( \sigma_n < \lambda_n ) \inf_{y : \| y \| \geq r} f (y) .\]
So
\begin{align*}
 \Pr \left( \inf_{m \geq n} \| X_m \| \leq r_0 \right)
& \geq \Pr (\lambda_n < \infty) \geq \Pr (\lambda_n < \sigma_n )  \geq 1 - \frac{\Exp f(X_n)}{\inf_{y : \| y \| \geq r} f (y) } .\end{align*}
Since $r > r_0$ was arbitrary, and $\inf_{y : \| y \| \geq r} f (y) \to \infty$ as $r \to \infty$,
it follows that, for fixed $n \in \ZP$,
$\Pr ( \inf_{m \geq n} \| X_m \| \leq r_0  ) = 1$. 
Since this holds for all $n \in \ZP$, the result follows.
\end{proof}

The corresponding transience result is based on Theorem~3.5.6 of~\cite{mpw}.

\begin{lemma}
\label{lem:transience}
Let $X_0, X_1, \ldots$ be a stochastic process on $\R^d$ adapted to a filtration $\cF_0, \cF_1, \ldots$.
Let $f : \R^d \to \RP$ be such that $\sup_x f(x) < \infty$, $f(x) \to 0$ as $\| x \| \to \infty$, and $\inf_{x : \| x \| \leq r} f(x) >0$
for all $r \in \RP$. Suppose that there
exists $r_0 \in \RP$ for which, for all $n \in \ZP$,
\begin{align*}
\Exp [ f (X_{n+1} ) - f(X_n) \mid \cF_n] & \leq 0, \text{ on } \{ \| X_n \| \geq r_0 \}.\end{align*}
Then if $\Pr ( \limsup_{n \to \infty} \| X_n \| = \infty ) =1$, we have that $\Pr ( \lim_{n \to \infty} \| X_n \| = \infty ) = 1$.
\end{lemma}
\begin{proof}
Since $f$ is bounded, $\Exp f(X_n) < \infty$ for all $n$. Fix $n \in \ZP$ and $r_1 \geq r_0$.
For $r \in \ZP$ let $\sigma_r := \min \{ n \in \ZP : \| X_n \| \geq r \}$.
Since $\Pr ( \limsup_{n \to \infty} \| X_n \| = \infty ) =1$, we have $\sigma_r < \infty$, a.s.
Let $\lambda_{r} := \min \{ n \geq \sigma_r : \| X_n \| \leq r_1 \}$.
Then $f(X_{n \wedge \lambda_r} )$, $n \geq \sigma_r$, is a non-negative supermartingale,
which converges, on $\{ \lambda_r < \infty \}$, to $f(X_{\lambda_r} )$. By optional stopping
(e.g.~Theorem~2.3.11 of~\cite{mpw}), a.s.,
\[ 
\sup_{x : \| x \| \geq r} f(x) \geq f (X_{\sigma_r}) \geq \Exp [ f(X_{\lambda_r} ) \mid \cF_{\sigma_r} ] \geq \Pr ( \lambda_r < \infty \mid \cF_{\sigma_r} ) \inf_{x : \| x \| \leq r_1} f(x)  .\]
So
\begin{align*}
\Pr ( \lambda_r < \infty ) \leq \frac{\sup_{x : \| x \| \geq r} f(x)}{\inf_{x : \| x \| \leq r_1} f(x) } ,\end{align*}
which tends to $0$ as $r \to \infty$, by our hypotheses on $f$.
Thus,
\[ \Pr \left( \liminf_{n \to \infty} \| X_n \| \leq r_1 \right) = \Pr \left( \cap_{r \in \ZP} \left\{ \lambda_r < \infty \right\} \right)
= \lim_{r \to \infty} \Pr ( \lambda_r < \infty ) = 0 .\]
Since $r_1 \geq r_0$ was arbitrary, we get the result.
\end{proof}

Now we can complete the proof of Theorem~\ref{thm:opposite},
which includes Theorem~\ref{thm:normal-recurrence} as the special case $\alpha=0$.

\begin{proof}[Proof of Theorem~\ref{thm:opposite}.]
Let $\beta = \max (\beta^+,\beta^-)$, and recall the definition
of $\bc$ from~\eqref{eq:betac} and that of $s_0$ from~\eqref{eq:p-def-opposite}.
Suppose first that $0 \leq \beta < 1 \wedge \bc$. Then $s_0 >0$ and we may (and do) choose
$w \in (0,2s_0)$.
Also, take $\gamma \in (0,1)$; note $0 < \gamma w < 1$.
Consider the function $f_w^\gamma$ with $\theta_0 = \theta_1$ given by~\eqref{eq:theta-1}. 
 Then from~\eqref{eq:f-increment-interior}, we see that 
there exist $c >0$ and $r_0 < \infty$ such that, for all $x \in S_I$,
\begin{align}
\label{eq:f-super}
\Exp [ f_w^\gamma  ( \xi_{n+1} ) - f_w^\gamma  ( \xi_n ) \mid \xi_n = x ] 
 \leq - c \| x \|^{\gamma w -2}, \text{ for all } \| x \| \geq r_0 .
\end{align}
By choice of $w$, we have 
 $\beta - (1-w)\bc < 0$,
so~\eqref{eq:f-increment-boundary-small-beta} shows that, 
for all $x \in S_B^\pm$,
\[ 
\Exp [ f_w^\gamma  ( \xi_{n+1} ) - f_w^\gamma  ( \xi_n ) \mid \xi_n = x ] 
\leq - c  \| x \|^{\gamma w -2 +\beta^\pm}
,\]
	for some $c >0$ and all $\| x \|$ sufficiently large. In particular, this means that~\eqref{eq:f-super}
	holds throughout $S$. On the other hand, it follows from~\eqref{eq:f-crude-bound}
	and~\eqref{ass:bounded-moments} that there is a constant $C < \infty$ such that
	\begin{align}
\label{eq:f-exceptional}
\Exp [ f_w^\gamma  ( \xi_{n+1} ) - f_w^\gamma  ( \xi_n ) \mid \xi_n = x ] 
 \leq C, \text{ for all } \| x \| \leq r_0 .
\end{align}
Since $w, \gamma >0$, we have that $ f_w^\gamma (x) \to \infty$
	as $\| x \| \to \infty$. Then by Lemma~\ref{lem:recurrence} with the conditions~\eqref{eq:f-super}
	and~\eqref{eq:f-exceptional} and assumption~\eqref{ass:non-confinement}, we establish recurrence.

Next suppose that $\bc < \beta < 1$.
If $\beta^+ = \beta^- = \beta$, we use the function~$f_w^\gamma$, again with
$\theta_0 = \theta_1$ given by~\eqref{eq:theta-1}. 
We may (and do) choose $\gamma \in (0,1)$ and $w < 0$
with $w > - 2|s_0|$ and $\gamma w > w > 2-p$. By choice of $w$, we have $\beta - (1-w)\bc > 0$.
We have from~\eqref{eq:f-increment-interior} and~\eqref{eq:f-increment-boundary-small-beta} 
that~\eqref{eq:f-super} holds in this case also,
but now $f_w^\gamma (x) \to 0$ as $\| x \| \to \infty$, since $\gamma w <0$. Lemma~\ref{lem:transience} then gives transience when $\beta^+=\beta^-$.

Suppose now that $\bc < \beta < 1$ with $\beta^+ \neq \beta^-$. 
Without loss of generality, suppose that $\beta = \beta^+ > \beta^-$. We now use the function $F_w^{\gamma,\nu}$
defined at~\eqref{eq:F-def},
where, as above, we take $\gamma \in (0,1)$ and $w \in (-2|s_0|,0)$,
and we choose the constants $\lambda, \nu$ with $\lambda <0$ and $\gamma w + \beta^- -2 < 2\nu < \gamma w + \beta^+ -2$. Note that $2\nu < \gamma w -1$, so
$F_w^{\gamma,\nu} (x) = f_w^\gamma (x) ( 1 + o(1))$. With $\theta_0 = \theta_1$ given by~\eqref{eq:theta-1}, and this choice of~$\nu$, Lemma~\ref{lem:F-increments} applies.
The choice of $\gamma$ ensures that the right-hand side of~\eqref{eq:F-increment-interior} is eventually negative,
and the choice of $w$ ensures the same for~\eqref{eq:F-increment-boundary-beta-plus}. Since $\lambda <0$, the right-hand side
of~\eqref{eq:F-increment-boundary-beta-minus} is also eventually negative. Combining these three estimates shows, for all $x \in S$ with $\| x \|$ large enough,
\[ \Exp [ F_w^{\gamma,\nu} (\xi_{n+1} ) - F_w^{\gamma,\nu} (\xi_n) \mid \xi_n = x ] \leq 0 .\]
Since $F_w^{\gamma,\nu} (x) \to 0$ as $\| x \| \to \infty$,  Lemma~\ref{lem:transience} gives transience.

Of the cases where $\beta^+, \beta^- <1$, it remains to consider the borderline
case where $\beta = \bc \in (0,1)$. Here Lemma~\ref{lem:g-gamma-increments}
together with Lemma~\ref{lem:recurrence} proves recurrence.
Finally, if $\beta^+, \beta^- > 1$, we apply Lemma~\ref{lem:w-increments}
together with Lemma~\ref{lem:recurrence} to obtain recurrence. Note that both of these critical cases
require~\eqref{ass:zero-drift-plus} and~\eqref{ass:covariance-plus}.
\end{proof}

Next we turn to moments of passage times: we prove
Theorem~\ref{thm:opposite-moments}, which includes Theorem~\ref{thm:normal-moments} as
the special case $\alpha=0$.
Here the criteria we apply
are from~\cite[\S 2.7]{mpw}, which are heavily based on those from~\cite{aim}.

\begin{proof}[Proof of Theorem~\ref{thm:opposite-moments}.]
Again let $\beta = \max (\beta^+,\beta^-)$.
First we prove the existence of moments part of~(a)(i).
Suppose that $0 \leq \beta < 1 \wedge \bc$, 
so $s_0$ as defined at~\eqref{eq:p-def-opposite}
satisfies $s_0 > 0$.
We use the function $f_w^\gamma$,
with $\gamma \in (0,1)$ and $w \in (0,2s_0)$ as in the first part of the proof of Theorem~\ref{thm:opposite}.
We saw in that proof 
that for these choices of $\gamma, w$ we have that~\eqref{eq:f-super} holds for all $x \in S$.
 Rewriting this slightly,
using the fact that $f_w^\gamma (x)$ is bounded above and below by constants
times $\| x\|^{\gamma w}$ for all $\| x\|$ sufficiently large,
we get that 
there are constants $c >0$
and
$r_0 < \infty$ for which
\begin{equation}
\label{eq:f-recurrence-quantitative}
 \Exp [f_w^\gamma ( \xi_{n+1} ) - f_w^\gamma ( \xi_n ) \mid \xi_n = x ] 
\leq -c ( f_w^\gamma ( x ) )^{1-\frac{2}{\gamma w}}, \text{ for all $x \in S$ with $\| x \| \geq r_0$} . \end{equation}
Then we may apply Corollary~2.7.3 of~\cite{mpw} to get
$\Exp_x ( \tau_r^s ) < \infty$
for any $r \geq r_0$ and any $s < \gamma w /2$.
Taking $\gamma <1$ and $w < 2s_0$ arbitrarily close to their upper bounds,
we get $\Exp_x ( \tau_r^s ) < \infty$ for all $s < s_0$.

Next suppose that $0 \leq \beta \leq \bc$. Let $s > s_0$.
First consider the case where $\beta^+ = \beta^-$. Then we consider $f_w^\gamma$
with $\gamma > 1$, $w > 2s_0$ (so $w>0$), and $0 < w \gamma < 2$.
Then, since $\beta - (1-w) \bc = \bc - \beta +(w-2s_0) \bc >0$,
we have from~\eqref{eq:f-increment-interior} and~\eqref{eq:f-increment-boundary-small-beta} that
\begin{equation}
\label{eq:f-sub}  \Exp [f_w^\gamma ( \xi_{n+1} ) - f_w^\gamma ( \xi_n ) \mid \xi_n = x ] \geq 0,
\end{equation}
for all $x \in S$ with $\| x \|$ sufficiently large.
Now set $Y_n := f_w^{1/w} ( \xi_n)$,
and note that $Y_n$ is bounded above and below by constants times $\| \xi_n\|$, and
 $Y_n^{\gamma w} = f_w^\gamma (\xi_n)$. Write $\cF_n = \sigma ( \xi_0, \xi_1, \ldots, \xi_n)$.
Then we have shown in~\eqref{eq:f-sub}  that
\begin{equation}
\label{eq:non-existence1}
 \Exp [ Y^{\gamma w}_{n+1} - Y^{\gamma w}_n \mid \cF_n ] \geq 0, \text{ on } \{  Y_n > r_1 \} ,\end{equation}
for some $r_1$ sufficiently large. 
Also, from the $\gamma =1/w$ case of~\eqref{eq:f-increment-interior} and~\eqref{eq:f-increment-boundary-small-beta},
\begin{equation}
\label{eq:non-existence2}
\Exp [ Y_{n+1} - Y_n \mid \cF_n ] \geq - \frac{B}{Y_n}, \text{ on } \{  Y_n > r_2 \} ,\end{equation}
for some $B < \infty$ and $r_2$ sufficiently large. (The right-hand side of~\eqref{eq:f-increment-boundary-small-beta}
is still eventually positive, while the right-hand-side of~\eqref{eq:f-increment-interior}
will be eventually negative if $\gamma < 1$.) 
Again let $E_x = \{ \| \Delta \| < \| x \|^\delta \}$ for $\delta \in (0,1)$.
Then from the $\gamma =1/w$ case of~\eqref{eq:h-w-taylor-truncated},
\[ \left| f_w^{1/w} (\xi_1 ) - f_w^{1/w} (\xi_0 ) \right|^2 \2 {E_x} \leq C \| \Delta \|^2 ,\]
while from the $\gamma =1/w$ case of~\eqref{eq:f-crude-bound}  we have
\[ \left| f_w^{1/w} (\xi_1 ) - f_w^{1/w} (\xi_0 ) \right|^2 \2 {E_x^\rc } \leq C \| \Delta \|^{2/\delta} .\]
Taking $\delta \in (2/p,1)$, it follows from~\eqref{ass:bounded-moments} that for some $C<\infty$, a.s.,
\begin{equation}
\label{eq:non-existence3}
 \Exp [ ( Y_{n+1} - Y_n  )^2 \mid \cF_n ] \leq C .\end{equation}
The three conditions~\eqref{eq:non-existence1}--\eqref{eq:non-existence3}
show that we may apply Theorem~2.7.4 of~\cite{mpw}
to get $\Exp_x ( \tau_r^s ) = \infty$ for all $s > \gamma w/2$, all $r$ sufficiently large, and all $x \in S$ with $\| x \| > r$. Hence, taking $\gamma >1$ and $w > 2s_0$
arbitrarily close to their lower bounds, we get $\Exp_x ( \tau_r^s ) = \infty$ for  all $s>s_0$ and appropriate $r, x$.
This proves the non-existence of moments part of~(a)(i) in the case $\beta^+ = \beta^-$.

Next suppose that $0 \leq \beta^+, \beta^- \leq \bc$ with $\beta^+ \neq \beta^-$.
Without loss of generality, suppose that $0 \leq \beta^- < \beta^+ = \beta \leq \bc$.
Then $0 \leq s_0 < 1/2$.
We consider the function $F_w^{\gamma,\nu}$ given by~\eqref{eq:F-def} with $\theta_0 = \theta_1$ given by~\eqref{eq:theta-1},
$\lambda > 0$, $w \in (2s_0, 1)$, and
$\gamma >1$ such that $\gamma w < 1$.
Also, take $\nu$ for which
$\gamma w + \beta^- -2 < 2\nu < \gamma w + \beta^+ -2$. Then by choice of
$\gamma$ and $w$, we have that the right-hand sides of~\eqref{eq:F-increment-interior} and~\eqref{eq:F-increment-boundary-beta-plus} are both eventually positive. Since $\lambda > 0$, 
the right-hand side of~\eqref{eq:F-increment-boundary-beta-minus}
is also eventually positive. Thus
\[ \Exp [ F_w^{\gamma,\nu} ( \xi_{n+1} ) - F_w^{\gamma,\nu} (\xi_n ) \mid \xi_n = x] \geq 0 ,\]
for all $x \in S$ with $\| x \|$ sufficiently large. Take $Y_n := ( F_w^{\gamma,\nu} (\xi_n) )^{1/(\gamma w)}$.
Then we have shown that, for this $Y_n$, the condition~\eqref{eq:non-existence1} holds.
Moreover, since $\gamma w <1$ we have from convexity that~\eqref{eq:non-existence2} also holds.
Again let $E_x = \{ \| \Delta \| <\| x \|^\delta \}$.
From~\eqref{eq:h-w-taylor-truncated} and~\eqref{eq:v-taylor},  
\[ \left| F^{\gamma,\nu}_w (x +y ) - F^{\gamma,\nu}_w (x) \right| \leq C \| y \| \| x\|^{\gamma w -1} ,\]
for all $y \in B_{r/2} (x)$. 
Then, by another Taylor's theorem calculation,
\begin{align*}
\left| \bigl(  F^{\gamma,\nu}_w (x +y )  \bigr)^{1/(\gamma w)} - \bigl(  F^{\gamma,\nu}_w (x  )  \bigr)^{1/(\gamma w)} \right|
\leq C \| y \| , \end{align*}
for all $y \in B_{r/2} (x)$. It follows that
$\Exp_x  [ ( Y_1 - Y_0 )^2 \2 { E_x }  ] \leq C$.
Moreover, by a similar argument to~\eqref{eq:f-gamma-big-jump},
$| Y_1 - Y_0 |^2 \leq C \| \Delta \|^{2\gamma w/\delta}$ on $E_x^\rc$,
so taking $\delta \in (2/p,1)$ and using the fact that $\gamma w <1$,
we get
$\Exp_x  [ ( Y_1 - Y_0 )^2 \2 { E^\rc_x }  ] \leq C$ as well.
Thus we also verify~\eqref{eq:non-existence3} in this case. 
Then we may again apply Theorem~2.7.4 of~\cite{mpw}
to get $\Exp_x ( \tau_r^s ) = \infty$ for all $s > \gamma w/2$, and hence all $s>s_0$. This completes the proof of~(a)(i).

For part~(a)(ii), suppose first that $\beta^+ = \beta^- = \beta$, and 
that $\bc \leq \beta <1$. 
We apply the function $f_w^\gamma$ with $w >0$ and $\gamma > 1$.
Then we have from~\eqref{eq:f-increment-interior} and~\eqref{eq:f-increment-boundary-small-beta} that~\eqref{eq:f-sub}
holds. Repeating the argument below~\eqref{eq:f-sub} shows that $\Exp_x ( \tau_r^s ) = \infty$ for all $s > \gamma w/2$, and hence all $s>0$. The case where $\beta^+ \neq \beta^-$
is similar, using an appropriate $F_w^{\gamma,\nu}$. This proves~(a)(ii).

It remains to consider the
case 
where $\beta^+, \beta^- >1$. Now we apply $f_w^\gamma$ with $\gamma >1$ and $w \in (0,1/2)$ small enough,
noting Remark~\ref{rem:w-choice}.
In this case~\eqref{eq:f-increment-interior} with~\eqref{eq:f-increment-boundary-big-beta} and Lemma~\ref{lem:facts}
show that~\eqref{eq:f-sub}
holds, and repeating the argument below~\eqref{eq:f-sub} shows that $\Exp_x ( \tau_r^s ) = \infty$ for all $s>0$.
This proves part~(b).
\end{proof}

\appendix

\section{Properties of the threshold function}
\label{sec:threshold}

For a constant $b \neq 0$, consider the function
\[ \phi (\alpha ) = \sin^2 \alpha + b \sin 2 \alpha .\]
Set $\alpha_0 := \frac{1}{2} \arctan ( -2b)$, which has $0 < | \alpha_0 | < \pi/4$.

\begin{lemma}
\label{lem:calculus}
There are two stationary points of $\phi$
 in $[-\frac{\pi}{2},\frac{\pi}{2}]$.
One of these is a local minimum at $\alpha_0$, with
\[ \phi (\alpha_0 ) = \frac{1}{2} \left( 1 - \sqrt{1 + 4 b^2} \right) < 0.\]
The other is a local maximum,  at $\alpha_1 = \alpha_0 + \frac{\pi}{2}$
if $b>0$, or at  $\alpha_1 = \alpha_0 - \frac{\pi}{2}$ if $b < 0$, with
\[ \phi (\alpha_1) = \frac{1}{2} \left( 1 + \sqrt{1 + 4 b^2} \right) >1 .
\]
\end{lemma}
\begin{proof}
We compute $\phi' (\alpha ) =   \sin 2 \alpha + 2 b \cos 2 \alpha$
and $\phi''(\alpha) = 2  \cos 2 \alpha - 4b \sin 2 \alpha$.
Then $\phi'(\alpha) =0$ if and only if
$\tan 2 \alpha = -2b$. Thus the stationary values
of $\phi$ are $\alpha_0 + k \frac{\pi}{2}$, $k \in \Z$.
Exactly two of these values fall in $[-\frac{\pi}{2},\frac{\pi}{2}]$,
namely $\alpha_0$ and $\alpha_1$ as defined in the statement of the lemma.
Also
\[ \phi'' (\alpha_0 ) = 2   \cos 2 \alpha_0 - 4 b \sin 2 \alpha_0 =
 \left( 2   +8b^2 \right) \cos 2\alpha_0 > 0,\]
so $\alpha_0$ is a local minimum.
Similarly, if $| \delta | = \pi/2$, then $\sin 2 \delta = 0$ and $\cos 2 \delta = -1$, so
\[ \phi'' (\alpha_0 + \delta ) = -  \cos 2 \alpha_0 + 4b \sin 2 \alpha_0 
= - \phi''(\alpha_0) ,\]
and hence the  stationary point at   $\alpha_1$ is a local maximum.
Finally, to evaluate the values of $\phi$ at the stationary points, note that
\[ \cos 2\alpha_0 = \frac{1}{\sqrt{1 + 4 b^2}} , \text{ and }
\sin 2 \alpha_0 = \frac{-2b}{\sqrt{1 + 4 b^2}} ,\]
and use the fact that  $2 \sin^2 \alpha_0 =  1- \cos 2 \alpha_0$
to get $\phi (\alpha_0)$,
and that $2 \cos^2 \alpha_0 = \cos 2 \alpha_0 +1$ to get
$\phi (\alpha_1 ) = \cos^2 \alpha_0 - b \sin 2 \alpha_0 =
1 - \phi (\alpha_0)$.
\end{proof}

\begin{proof}[Proof of Proposition~\ref{prop:bc-max-min}.]
By Lemma~\ref{lem:calculus} (and considering separately the
case $\sigma_1^2 = \sigma_2^2$) we see that the extrema of $\bc (\Sigma,\alpha)$
over $\alpha \in [ - \frac{\pi}{2}, \frac{\pi}{2}]$ are
\begin{align*}
 \frac{\sigma_1^2+\sigma_2^2}{2\sigma_2^2} \pm \frac{1}{2\sigma_2^2} \sqrt{ \left( \sigma_2^2 - \sigma_1^2 \right)^2 + 4 \rho^2 },
 \end{align*}
 as claimed at~\eqref{eq:bc-max-min}. It remains to show that the minimum is strictly positive,
 which is a consequence of the fact that
 \[ \sigma_1^2 + \sigma_2^2 - \sqrt{ \left(\sigma_1^2 + \sigma_2^2 \right)^2 - 4 \left( \sigma_1^2 \sigma_2^2 -  \rho^2 \right) } > 0 ,\]
 since $\rho^2 < \sigma_1^2 \sigma_2^2$ (as $\Sigma$ is positive definite).
\end{proof}

\section*{Acknowledgements}

A.M.~is supported by The Alan Turing Institute under the EPSRC grant EP/N510129/1 and
by the EPSRC grant EP/P003818/1 and the Turing Fellowship funded by the Programme on Data-Centric Engineering of Lloyd's Register Foundation. Some of the work reported in this paper was undertaken during visits by M.M.~and A.W.~to The Alan Turing Institute,
whose hospitality is gratefully acknowledged. The authors also thank two referees for helpful comments and suggestions.

\end{document}